\documentclass{amsart}

\usepackage{amsthm,amssymb,latexsym,amsmath}
\usepackage[all]{xy}
\usepackage[english]{babel}
\usepackage[utf8x]{inputenc}
\usepackage[dvipsnames]{xcolor}
\usepackage[normalem]{ulem}
\usepackage[colorlinks,linkcolor=blue,citecolor=blue,filecolor=blue,urlcolor=blue]{hyperref}

\newcommand{\p}[1]{{\mathbb{P}^{#1}}}

\newcommand{\op}[1]{{\mathcal O}_{\mathbb{P}^{#1}}}

\newcommand{\pr}{\operatorname{pr}\nolimits}

\newcommand{\cale}{{\mathcal E}}

\newcommand{\calh}{{\mathcal H}}
\newcommand{\cali}{{\mathcal I}}

\newcommand{\calo}{{\mathcal O}}

\DeclareMathOperator{\coker}{coker}

\DeclareMathOperator{\codim}{{codim}}
\DeclareMathOperator{\rk}{{rk}}

\newcommand\dual{{\scriptscriptstyle{\vee}}}

\newlength{\rrrr}

\newcommand{\intoo}[1]{\:
\xymatrix@1{\ar@{^(->}[r]^{#1}&}\:}
\newcommand{\ontoo}[1]{\:
\xymatrix@1{\ar@{->>}[r]^{#1}&}\:}

\newcommand\Ext{\operatorname{Ext}\nolimits}

\newtheorem{theorem}{Theorem}

\newtheorem{corollary}[theorem]{Corollary}

\newtheorem{remark}[theorem]{Remark}

\numberwithin{equation}{section}
\numberwithin{theorem}{section}

\begin{document}

\title{Series of rational moduli components of stable rank 2 vector bundles on $\mathbb{P}^3$}

\author{Alexey A. Kytmanov}
\address{Institute of Space and Information Technology\\
Siberian Federal University\\ 
79 Svobodny Avenue\\
660041 Krasnoyarsk, Russia}
\email{aakytm@gmail.com}

\author{Alexander S. Tikhomirov}
\address{Department of Mathematics\\
National Research University  
Higher School of Economics\\
6 Usacheva Street\\ 
119048 Moscow, Russia}
\email{astikhomirov@mail.ru}

\author{Sergey A. Tikhomirov}
\address{Department of Physics and Mathematics,
Yaroslavl State Pedagogical University named after K.D.Ushinskii,
108 Respublikanskaya Street,
150000 Yaroslavl, Russia\\ 
${\ \ \ \ }$Koryazhma Branch of Northern (Arctic)
Federal University\\
9 Lenin Avenue\\
165651 Koryazhma, Russia}
\email{satikhomirov@mail.ru}

\begin{abstract}
We study the problem of rationality of an infinite series of components, the so-called Ein components, of the Gieseker-Maruyama moduli space $M(e,n)$ of rank 2 stable vector bundles with the first Chern class $e=0$ or -1 and all possible values of the second Chern class $n$ on the projective 3-space. The generalized null correlation bundles constituting open dense subsets of these components are defined as cohomology bundles of monads whose members are direct sums of line bundles of degrees depending on nonnegative integers $a,b,c$, where $b\ge a$ and $c>a+b$. We show that, in the wide range when $c>2a+b-e,\ b>a,\ (e,a)\ne(0,0)$, the Ein components are rational, and in the remaining cases they are at least stably rational. As a consequence, the union of the spaces $M(e,n)$ over all $n\ge1$ contains an infinite series of rational components for both $e=0$ and $e=-1$. Explicit constructions of rationality of Ein components under the above conditions on $e,a,b,c$ and, respectively, of their stable rationality in the remaining cases, are given. In the case of rationality, we construct universal families of generalized null correlation bundles over certain open subsets of Ein components showing that these subsets are fine moduli spaces. As a by-product of our construction, for $c_1=0$ and $n$ even, they provide, perhaps the first known, examples of fine moduli spaces not satisfying the condition ``$n$ is odd'', which is a usual sufficient condition for fineness. 



\end{abstract}

\maketitle

\section{Introduction}\label{section 1}

For $e\in\{-1,0\}$ and $n\in\mathbb{Z}_+$, let $M(e,n)$ be the Gieseker-Maruyama moduli space of stable rank 2 algebraic vector bundles with Chern classes $c_1=e,\ c_2=n$ on the projective space $\mathbb{P}^3$. R.~Hartshorne \cite{H-vb} showed that $M(e,n)$ is a quasi-projective scheme, nonempty for arbitrary $n\ge1$ in the case $e=0$ and, respectively, for even $n\ge2$ in the case $e=-1$, and the deformation theory predicts that each irreducible component of $M(e,n)$ has dimension at least $8n-3+2e$.  

In this paper we study the problem of rationality of irreducible 
components of $M(e,n)$. Since 70ies not so much has been known 
about it. In particular, in the case $e=0$, it is known 
(see \cite{H-vb}, \cite{ES}, \cite{B1}, \cite{CTT}, \cite{T1}, 
\cite{T2}) that the scheme $M(0,n)$ contains an irreducible 
component $I_n$ of the expected dimension $8n-3$, and this 
component is the closure of the open subset of $M(0,n)$ 
constituted by the so-called  mathematical instanton vector 
bundles. Furthermore, according to the recent result of 
\cite[Theorem 3]{TTV}, $M(0,n)$ contains, besides $I_n$, at least 
one more irreducible component for any $n\ge146$. Next, $M(0,n)$ 
is irreducible (hence coincides with $I_n$) and rational for 
$n=1, 2$ \cite{H-vb}. The rationality of $I_3$ and of $I_5$ was 
proved in \cite{ES} and \cite{K}, respectively, and for $n=4$ and 
$n\ge6$ the rationality of $I_n$ is still a challenging open 
question. Note that $M(0,n)$ is reducible for $n\ge3$, and the 
exact number of irreducible components of $M(0,n)$ is nowadays 
known only up to $n=5$ \cite{AJTT}. We list these components in 
Section \ref{section 8}. 

In the case $e=-1$, for each $n\ge1$, the space $M(-1,2n)$ contains at least one irreducible component $Y_{2n}$ of the expected dimension $16n-5$ \cite{H-vb}. In particular, $M(-1,2)=Y_2$ is a rational variety of the expected dimension 11 by \cite{HS}. The space $M(-1,4)$ is also known -- it contains, besides the rational component $Y_4$ of the expected dimension 27, one more rational component of dimension 28.
For $n\ge6$ the exact number of irreducible components of $M(0,n)$ is still unknown (see details in Section \ref{section 8}).

In 1978 W.~Barth and K.~Hulek \cite{BH} found, for each integer 
$k\ge1$, a rational $(3k^2+10k+8)$-dimensional family 
$\tilde{Q}_k$ of vector bundles from $M(0,2k+1)$, and 
G.~Ellingsrud and S.~A.~Str\o mme 
in \cite[(4.6)--(4.7)]{ES} showed that the image of 
$\tilde{Q}_k$ under the modular morphism $\tilde{Q}_k\to 
M(0,2k+1)$ is an open subset of an irreducible component $Q_k$ 
distinct from the instanton component $I_{2k+1}$. Besides, from 
the definition of $Q_k$, it follows that it is (at least) 
unirational. Later in 1984, V.~K.~Vedernikov \cite{V1} 
constructed, for $1\le l\le k$, a family $V_1(k,l)$ of bundles from $M(0,n_1)$;  for $1\le2l\le k$, a family $V_2(k,l)$ of bundles from $M(0,n_2)$; for $1\le2l\le k+2$, a family $V_3(k,l)$ of bundles from $M(-1,n_3)$, where 
$n_1,n_2,n_3$ are certain polynomials on $k,l$. In his 
subsequent paper \cite{V2}, one more family $V_4(k)$ of bundles from $M(0,(k+1)^2)$ was found for $k\ge1$. In \cite{V1}, \cite{V2}, the constructions of stable rationality of $V_1(k,l)$ and of rationality of $V_2(k,l)$ and $V_4(k)$ were given, respectively (see Remark \ref{V(k,l)} below for details). Besides, the author asserted that these families are open subsets of irreducible components of $M(e,n)$, though the proofs for these statements were not given. A more general series of rank 2 bundles depending on triples of integers $a,b,c$, appeared in 1984 in the paper of A. Prabhakar 
Rao \cite{Rao} (cf. Remark \ref{big families}). Soon after that, 
in 1988, L.~Ein \cite{Ein} independently studied these bundles 
(called in his paper the generalized null correlation bundles) 
and proved that they constitute open subsets of irreducible components of $M(e,n)$ (called below Ein components). Surprisingly, Ein components contain Vedernikov's families $V_1(k,l)$ and $V_4(k)$, respectively, $V_2(k,l)$ and  $V_3(k,l)$ as their open subsets in special cases when $e=a=0$, respectively, $a=b$ (see details in Remark \ref{V(k,l)}). Moreover, when $e=a=0,\ b=k\ge1,\ c=k+1$, the closure of Vedernikov's family $V_1(k,1)$ coincides with the component $Q_k$  of Ellingsrud-Str\o mme, i. e. $Q_k$ is also an Ein component.

The problem of rationality of Ein components is the main subject 
of this paper. We will prove their rationality in a wide range of 
parameters $a,b,c$ when $c>2a+b-e,\ b>a,\ (e,a)\ne(0,0)$, and 
their (at least) stable rationality in the remaining cases. In 
particular, we show that our results cover Vedernikov's results 
in the case of $e=0,\ a=b>0,\ c>3a$ and improve them in the case 
of $e=a=0,\ b>0$ (see Remark \ref{V(k,l)}). Together with the 
remaining Vedernikov's results, this gives a complete solution to 
the problem of rationality or, otherwise, (at least) stable 
rationality of Ein components for all possible values of 
$e,a,b,c$. Before proceeding to precise formulations, we recall 
briefly the definition of generalized null correlation bundles.

For integers $a,b,c$ with $b\ge a\ge0,\ c>a+b$, consider the monad
\begin{equation}\label{ein monad}
0\to\op3(-c+e)\to\calh\to\op3(c) \to 0,
\end{equation}
where
\begin{equation}\label{calh}
\calh= \op3(a) \oplus \op3(-a+e) \oplus \op3(b) \oplus \op3(-b+e),
\end{equation}
such that the cohomology sheaf $E$ of this monad is locally free. According to \cite[Prop.~3.1]{Rao} (see also \cite[Prop.~1.2(a)]{Ein}), such monads exist and their cohomology rank 2 
vector bundle $E$ is stable. We call $E$ a \emph{generalized 
null correlation bundle} and denote by $N_{\mathrm{nc}}$ the 
set of all generalized null correlation bundles for the above 
integers $e,a,b,c$. Ein shows in \cite{Ein} that 
$N_{\mathrm{nc}}$ is a dense Zariski open subset of an 
irreducible component $\overline{N}(e,a,b,c)$ of the space $M(e,n)$, where $n=c^2-a^2-b^2-e(c-a-b)$. We therefore call these moduli components $\overline{N}(e,a,b,c)$ the \emph{Ein components} of $M(e,n)$.

We give now a sketch of the contents of the paper. In Section 
\ref{section 2}, we begin the study of the Ein component 
$\overline{N}(e,a,b,c)$ for any admissible $e,a,b,c$. We first 
describe a certain dense open subset $N$ specified by the 
behaviour of restrictions of generalized null correlation 
bundles from $N$ onto surfaces $S$ of the linear series 
$\mathrm{P}=|\op3(c-b)|$. (The precise definition of $N$ is 
given in \eqref{descriptn of N}). Using Quot-schemes, we then 
conststruct a certain principal PGL-bundle $Y\to N$ together 
with a family of generalized null correlation bundles over $Y$, 
and, respectively, a variety $X$ with a surjection $\theta: X\to 
N$ which is an open subfibration of some explicitely described 
projective fibration over $N$. These data yield a family 
$\mathbf{E}$ of generalized null correlation bundles over the 
variety $\mathbf{X}=X\times_NY$ induced by the aforementioned 
family. In Section \ref{section 3}, we relate to $\mathbf{E}$ a 
family $\mathbf{F}$ of rank-2 reflexive sheaves. These sheaves 
$F$ are obtained from bundles $E$ of the family $\mathbf{E}$ by 
elementary transformations $E\rightsquigarrow F$ along specially 
chosen surfaces $S$ of degree $c-b$. This is an analogue of the 
so-called reduction step procedure of R.~Hartshorne (cf. Remark 
\ref{reduction step}(i)). 

In Section \ref{section 4.0}, we provide a detailed enough plan 
of the proof of the main result of the paper --- Theorem 
\ref{main Thm} which states that $N$ is a rational variety and a 
fine open subset of the moduli component  $\overline{N}(e,a,b,c)$ if 
$c>2a+b-e,\ b>a,\ (e,a)\ne(0,0)$, and is at least stably 
rational otherwise. The idea is to construct and then relate two 
diagrams of varieties and projections: 
\begin{equation}\label{two diagrams}
W\xrightarrow{\pi}X\xrightarrow{\theta}N\ \ \ \text{and}\ \ \ 
V\xrightarrow{\lambda}T\xrightarrow{\mu}\mathrm{R}\xrightarrow{\mathrm{r}}\mathrm{P}.
\end{equation}
In these diagrams all the projections are open subfibrations of some locally trivial projective fibrations (see diagrams \eqref{diag for X,W} and \eqref{diag for T,B} for details). In particular, $V$ is rational and $W$ is birational to $N\times\mathbb{P}^k$ for certain $k\ge0$. We then relate the two diagrams in \eqref{two diagrams} by constructing an isomorphism 
\begin{equation}\label{isom V to W}
f:W\xrightarrow{\sim}V
\end{equation}
and its inverse morphism  $h=f^{-1}:V\xrightarrow{\sim}W$. On 
the level of sets the maps $f$ and $h$ are given by explicit 
formulas \eqref{set-theoretic def of f} and \eqref{set-theoretic 
def of h}. In a sense, these are just the above mentioned 
elementary transformation $E\rightsquigarrow F$ and its dual  
$F\rightsquigarrow E$. The isomorphism \eqref{isom V to W} then 
immediately yields Theorem \ref{main Thm}: the condition 
$c>2a+b-e,\ b>a,\ (e,a)\ne(0,0)$ by the dimension count leads to 
the isomorphism $\theta\circ\pi:W\xrightarrow{\sim}N$, so that 
$N\simeq V$ is rational; respectively, it is stably rational 
otherwise.

Our plan described in Section \ref{section 4.0} consists of four 
steps, which are developed in full detail in the subsequent 
Sections~\ref{section 4}--\ref{section 7}. In steps 1--3 which 
are performed in Sections \ref{section 4}, \ref{section 5}, and 
\ref{section 6}, we construct the varieties and the projections, 
respectively, $W\xrightarrow{\pi}X$, $T$, and 
$V\xrightarrow{\lambda}T\xrightarrow{\mu}\mathrm{R}\xrightarrow
{\mathrm{r}}\mathrm{P}$ involved in \eqref{two diagrams}. 
Besides, we build new families $\underline{\mathbf{E}}$ and 
$\underline{\mathbf{F}}$ of generalized null correlation bundles 
and, respectively, reflexive sheaves. The interplay between the 
two pairs of families $\mathbf{E}$, $\mathbf{F}$ and 
$\underline{\mathbf{E}}$, $\underline{\mathbf{F}}$ leads to the 
final step 4 of the proof of Theorem \ref{main isom} which is 
completed in Section~\ref{section 7}. Thus, the union of the 
spaces $M(e,n)$ over all $n\ge1$ contains an infinite series of 
rational components (see Corollary \ref{Cor}). As a by-product 
of Theorem~\ref{main Thm}, we show that, for $c_1=0$ and $n$ 
even, the open subsets $N$ of Ein components $\overline{N}$ 
provide, perhaps, the first known examples of fine moduli 
components of rank 2 stable bundles not satisfying the condition 
``$n$ is odd'' -- a usual sufficient condition for fineness (see 
Remark \ref{fine moduli}).
As another application of Theorem~\ref{main Thm}, in 
Section~\ref{section 8} we give a list of known irreducible 
components of $M(e,n)$, including Ein components, for small 
values of $n$, up to $n=20$, specify those of Ein components 
which are rational, respectively, stably rational, for both 
$e=0$ and $e=-1$, and give their dimensions.

\vspace{2mm}
\noindent
\textbf{Conventions and notation}.
\begin{itemize}
\item Everywhere in this paper we work over the base field $\mathbf{k}=\mathbf{\bar{k}}$ of characteristic 0.
\item $\mathbb{P}^3$ is the projective 3-space over $\mathbf{k}$.
\item Given a morphism of schemes $f:X\to Y$ and a coherent sheaf $\mathcal{F}$ on $\p3\times Y$, set 
\[
\mathcal{F}_X:=(\mathrm{id}_{\p3}\times f)^*\mathcal{F}.
\]
This notation will be systematically used throughout the paper. 
\item For any coherent sheaf $\mathcal{G}$ on a scheme $X$,
we set $\mathbb{P}(\mathcal{G}):=\mathrm{Proj} (S^{\cdot}_{\calo_{X}}\mathcal{G})$. Also, $\calo_Y(1)$ denotes the Grothendieck invertible sheaf on $Y=\mathbb{P}(\mathcal{G})$.
\item Given $m,n\in\mathbb{Z}$, $\mathbf{P}$ a projective space 
an arbitrary dimension, $X$ a scheme, and $\mathcal{A}$ a 
coherent sheaf on $\mathbf{P}\times\p3\times X$, set
\begin{equation}\label{Amn}
\mathcal{A}(m,n):=\mathcal{A}\otimes\calo_{\mathbf{P}}(n)
\boxtimes\calo_{\p3}(m)\boxtimes \calo_X,\ \ \ 
\mathcal{A}(m):=\mathcal{A}(m,0).
\end{equation}
\item $M(e,n)$ is the Gieseker-Maruyama moduli space of stable rank 2 algebraic vector bundles on $\mathbb{P}^3$, with Chern classes $c_1=e\in\{-1,0\},\ c_2=n\in\mathbb{Z}_+$ for $e=0$, respectively, $\in2\mathbb{Z}_+$ for $e=-1$.
\item $\overline{N}=\overline{N}(e,a,b,c)$ is the Ein component of the moduli space $M(e,n),\  b\ge a\ge0,\ c>a+b,\ n=c^2-a^2-b^2-e(c-a-b)$.
\item  $N_{\mathrm{nc}}$ is the open dense subset of $\overline{N}$ consisting of generalized null correlation bundles.
\item $N$ is the open dense subset of $N_{\mathrm{nc}}$ defined in \eqref{descriptn of N}.
\item For a stable rank 2 vector bundle $E$ with $c_1(E)=e,\ c_2(E)=n$ on $\p3$, we denote by $[E]$ its isomorphism class in $M(e,n)$. 
\item For a projective $\mathbb{P}^m$-fibration $p:X\to Y$, by
its \textit{open subfibration} we mean an open subset $U$ of $X$, together with the projection $p|_U:U\to Y$.
\end{itemize}

\noindent
\textbf{Acknowledgements}.
AAK was supported by the grant of the President of the Russian Federation for young scientists, project MD-197.2017.1.
AST was supported by a subsidy to the HSE from the Government of 
the Russian Federation for the implementation of Global 
Competitiveness Program. AST also acknowledges the support from 
the Max Planck Institute for Mathematics in Bonn, where this 
work was partially done during the winter of 2017.

\vspace{5mm}

\section{Ein component $\overline{N}(e,a,b,c)$ and its dense open subset $N$}\label{section 2}
 
\vspace{2mm}

In this Section, for an arbitrary Ein moduli component $\overline{N}=\overline{N}(e,a,b,c)$ we introduce a certain dense open subset $N$ of $\overline{N}$ which will be the main object of our study. We then construct a family $\mathbf{E}$ of generalized null correlation bundles on $\mathbb{P}^3$ with base $\mathbf{X}$ covering $N$ under the modular morphism $\mathbf{X}\to N,\ \mathbf{x}\mapsto\left[\mathbf{E}|_{\mathbb{P}^3\times\{\mathbf{x}\}}\right]$ (see Theorem \ref{Thm 2}). This family $\mathbf{E}$ will be used in subsequent sections.

Given integers $e,a,b,c$ with $e\in\{-1,0\}$ and $b\ge a\ge0,\ c>a+b$, consider the Ein component $\overline{N}=\overline{N}(e,a,b,c)$ of $M(e,n),\ n=c^2-a^2-b^2-e(c-a-b)$. As it is known from \cite[(2.2.B) and Section 3]{Ein} (see also \cite[Section 5]{B}), $\dim \overline{N}=h^0(\calh(c-e))-h^0(S^2\calh(-e))-1$.
Substituting here $\calh$ from (\ref{calh}), we obtain:
\begin{equation}\label{dim N 1}
\begin{split}
& \dim\overline{N}=
\binom{c+a-e+3}{3}+\binom{c+b-e+3}{3}+\binom{c-a+3}{3}\\
&+\binom{c-b+3}{3}-\binom{a+b-e+3}{3}-\binom{b-a+3}{3}-\binom{2a-e+3}{3}\\
&-\binom{2b-e+3}{3}-3-t(e,a,b), 
\end{split}
\end{equation}
where
\begin{equation}\label{t(e)}
t(0,a,b)=\left\{
\begin{array}{cl}
4, & \text{if}\ a=b=0, \\
1, & \text{if}\ 0=a<b\ \text{or}\ a=b>0, \\
0, & \text{otherwise}.
\end{array}\right. \quad
t(-1,a,b)=\left\{
\begin{array}{cl}
1, & \text{if}\ a=b, \\
0, & \text{otherwise}.
\end{array}\right.
\end{equation}
Consider the open dense subset $N_{\mathrm{nc}}$ of $\overline{N}$ consisting of generalized null correlation bundles. From (\ref{ein monad})--(\ref{calh}), we have
\begin{equation*}\label{h1E(m)}
h^1(E(m))=h^0(\calo_{\p3}(c+m))-h^0(\calo_{\p3}(a+m))-h^0(\calo_{\p3}(b+m)),\ \ \ m\le-1
\end{equation*}
for any bundle $[E]\in N_{\mathrm{nc}}$. In particular,
\begin{equation}\label{h1E(-c)}
h^1(E(-c))=1,
\end{equation}
\begin{equation}\label{h1E(-b)}
h^1(E(-b))=\binom{c-b+3}{3}-1,\ \ \ b>0.
\end{equation}

Consider more closely the monad \eqref{ein monad} with cohomology bundle $[E]\in N_{\mathrm{nc}}$:\\
\begin{equation}\label{lambda,mu}
\begin{split}
& 0\to \op3(-c+e) \stackrel{\lambda}{\longrightarrow}\calh
\stackrel{\mu}{\longrightarrow}\op3(c) \to 0, \ \ \ where\\
& \lambda=(f_2,-f_1,f_4,-f_3)^t,\ \ \ \mu=(f_1,f_2,f_3,f_4),\ \ \ \mu\circ\lambda=0,
\end{split}
\end{equation}
\begin{equation}\label{f1,...,f4}
\begin{split}
& f_1\in V_1:=H^0(\calo_{\p3}(c-a)),\ \ \ 
f_2\in V_2:=H^0(\calo_{\p3}(c+a-e)),\\ 
& f_3\in V_3:=H^0(\calo_{\p3}(c-b)),\ \ \ 
f_4\in V_4:=H^0(\calo_{\p3}(c+b-e)).
\end{split}
\end{equation}

Moreover, since $\mu$ is surjective, it follows that the 
subset $\cap_{i=1}^4\{f_i(x)=0\}$ of $\p3$ is empty. In 
particular, polynomials $f_1$ and $f_3$ do not have common 
factors of positive degree. This implies, in particular, that 
the surfaces 
\begin{equation}\label{S,S'}
S:=\{f_3(x)=0\} \ \ \ \mathrm{and}\ \ \ S':=\{f_1(x)=0\}
\end{equation}
intersect in a curve
\begin{equation}\label{C=S cap S'}
C_0:=S\cap S'.
\end{equation}
Note that, for the surface $S$ defined in \eqref{S,S'}, the equality
\begin{equation}\label{h0ES(-b)ne0}
h^0(E(-b)|_S)>0.
\end{equation}
holds. Indeed, the sheaf $K:=\ker(\mu)(-b)|_S$ satisfies the 
exact triple $0\to\calo_S(-c-b+e)\to K\to E|_S\to0$. Now, by the 
definition of $S$ the composition 
$\calo_S\overset{i}{\hookrightarrow} 
\mathcal{H}\xrightarrow{\mu(-b)|_S}\calo_S(c-b)$ is the zero 
morphism. Hence $i$ factors through a non-zero morphism 
$\calo_S\to K$, i. e. $h^0(K)\ne0$. Therefore, passing to 
sections in the above triple and using the vanishing of 
$h^0(\calo_S(-c-b+e))$, we obtain \eqref{h0ES(-b)ne0}.

Now consider the space $M=M(e,a,b,c)$ of monads \eqref{ein monad}:
\begin{equation}\label{def of M}
M:=\{(f_1,f_2,f_3,f_4)\in\Pi_{i=1}^4V_i\ |\  \eqref{lambda,mu}\ \text{is true and}\ \cap_{i=1}^4\{f_i(x)=0\}=\emptyset\}.
\end{equation}
There is a well-defined modular morphism
\[
\rho:\ M\twoheadrightarrow N_{\mathrm{nc}},\ \ \ (f_1,f_2,f_3,f_4)\mapsto[\ker(\mu)/\mathrm{im}(\lambda)].
\]
Clearly, $M$ is an open subset of the affine space
$\Pi_{i=1}^4V_i$, hence it is irreducible. Consider its dense
open subset
\begin{equation}\label{def of Ms}
\begin{split}
& M_s=\{(f_1,\ldots,f_4)\in M\ |\ \text{surface}\ S=\{f_3(x)=0\}\ \text{in}\ \eqref{S,S'}\\
& \text{and curve}\ C_0=\{f_1(x)=f_3(x)=0\}\ \text{in}\ \eqref{C=S cap S'}\ \text{are smooth}\}.
\end{split}
\end{equation}
Since $N_{\mathrm{nc}}$ is irreducible, there exists a dense 
open subset $N_s$ of $N_{\mathrm{nc}}$ contained in 
$\rho(M_s)$:
\begin{equation}\label{def of Ns}
\xymatrix{
&    \rho(M_s)\ar@{^{(}->}[dr]   &   \\
N_s\ \ar@{^{(}->}[rr]^-{\mathrm{dense\ open}}
\ar@{^{(}->}[ur] & & N_{\mathrm{nc}} }
\end{equation}
\begin{remark}\label{choice of Ns}
The choice of the subset $N_s$ satisfying \eqref{def of Ns} is not unique. From now on we, therefore, assume that, for each collection of admissible values of $e,a,b,c$, $N_s$ is a maximal (with respect to inclusion) such subset.
\end{remark}

Next, there exists a big enough positive integer $m$ such that all bundles from $N_{\mathrm{nc}}$ are $m$-regular in the sense of Mumford-Castelnuovo \cite[Section 4.3]{HL}. Let $\mathcal{P}\in\mathbb{Q}[x]$ be the Hilbert polynomial $\mathcal{P}(k)=\chi(E(k)),\ [E]\in N_{\mathrm{nc}}$, and let $\boldsymbol{\mathcal{B}}:=\mathbf{k}^{N_m}
\otimes\op3(-m)$, where $N_m:=\mathcal{P}(m)$. Consider the Quot-scheme $Q:=\mathrm{Quot}_{\p3}(\boldsymbol{\mathcal{B}}, \mathcal{P})$, together with the universal quotient morphism $\boldsymbol{\mathcal{B}}\boxtimes\calo_{Q}\twoheadrightarrow \mathbb{E}$. Then, the scheme 
\[
\mathcal{Y}=\Bigl\{y\in Q\ |\ \left[\mathbb{E}|_{\p3\times\{y\}}\right]\in N_s\Bigr\}
\]
is an open subscheme of $Q$, together  with a family
\begin{equation*}\label{family bbE}
\mathbb{E}_{\mathcal{Y}}=\mathbb{E}|_{\p3\times \mathcal{Y}}
\end{equation*}
of generalized null correlation bundles over $\mathcal{Y}$. Since all bundles from $N_s$ are stable, then, according to the GIT-construction \cite[Section 4.3]{HL} of $N_s$, the modular morphism
\begin{equation}\label{varphi}
\varphi:\ \mathcal{Y}\to N_s=\mathcal{Y}//G,\ \ y\mapsto\left[\mathbb{E}_\mathcal{Y}|_{\p3\times \{y\}}\right],
\ \ \ \ G=PGL(N_m),
\end{equation}
is a geometric $G$-quotient and a principal $G$-bundle. 

Since by Serre duality, for any $[E]\in N_s$ one has $h^2(E(c-e-4))=h^1(E(-c)),\ h^2(E(b-e-4))=h^1(E(-b))$, using (\ref{h1E(-c)}), (\ref{h1E(-b)}), and the base change we obtain that the sheaves 
\begin{equation}\label{L,L'}
L=R^2p_{2*}\mathbb{E}_{\mathcal{Y}}(c-e-4),\ \ \ 
L'=R^2p_{2*}\mathbb{E}_{\mathcal{Y}}(b-e-4), 
\end{equation}
where $p_2:\p3\times \mathcal{Y}\to \mathcal{Y}$ is the projection, are locally free $\calo_{\mathcal{Y}}$-sheaves of ranks 
\begin{equation}\label{rk Li}
\rk L=1,\ \ \ \mathbf{r}:=\rk L'=\binom{c-b+3}{3}-1.
\end{equation}

Consider the linear series 
\begin{equation}\label{bfP}
\mathbf{P}:=|\calo_{\p3}(c-b)|
\end{equation}
and its dense open subset
\begin{equation}\label{rmP}
\mathrm{P}:=\{S\in\mathbf{P}\ |\ S\ \text{is a smooth surface} \}.
\end{equation}
Let
\begin{equation}\label{Gamma}
\Gamma=\left\{(S,x)\in\mathbf{P}\times\p3\ |\ x\in S\right\}
\end{equation}
be the universal family of surfaces of degree $c-b$ in $\p3$. There
is an exact triple on $\mathbf{P}\times\p3\times \mathcal{Y}$:
\begin{equation}\label{exact0}
0\to\calo_{\mathbf{P}}(-1)\boxtimes\calo_{\p3}(b-c)
\boxtimes\calo_{\mathcal{Y}}\to\calo_{\mathbf{P}}\boxtimes
\calo_{\p3}\boxtimes\calo_{\mathcal{Y}}\to
\calo_{\Gamma\times \mathcal{Y}}\to0.
\end{equation}
Tensoring it with the sheaf 
$\mathbb{E}_{\mathcal{Y}}(c-e-4)\boxtimes\calo_{\mathbf{P}}$ 
and applying to the resulting exact triple the functor $R^ipr_{13*}$, where $pr_{13}:\mathbf{P}\times\p3\times \mathcal{Y}\to \mathbf{P}\times \mathcal{Y}$ is a projection, in view of the base change and the equalities $h^3\bigl(E(b-e-4)\bigr)=0$ we obtain an exact triple
\begin{equation}\label{exact1}
\calo_{\mathbf{P}}(-1)\boxtimes L'\overset{\psi}{\to}
\calo_{\mathbf{P}}\boxtimes L\to R^2pr_{13*}\left(\calo_{\mathbf{P}}\boxtimes
\mathbb{E}_{\mathcal{Y}}(c-e-4)|_{\Gamma\times \mathcal{Y}}\right)\to0.
\end{equation}
Now take an arbitrary point $y\in \mathcal{Y}$ and denote
\begin{equation*}\label{Ey}
[E_y]:=\varphi(y).
\end{equation*}
Resricting the triple (\ref{exact1}) onto $\mathbf{P}\times\{y\}$ and using (\ref{rk Li}) and the base change, we obtain an exact triple
\begin{equation}\label{exact2}
\mathbf{r}\calo_{\mathbf{P}}(-1)\xrightarrow{\psi\otimes
\mathbf{k}(y)}\calo_{\mathbf{P}}\to\coker\bigl(\psi\otimes
\mathbf{k}(y)\bigr)\to0,
\end{equation}
where by the base change we have for any surface $S\in\mathbf{P}$:
\begin{equation}\label{coker psi}
\coker\bigl(\psi\otimes\mathbf{k}(y)|_{\{(S,y)\}}\bigr)=
R^2pr_{23*}\left(\mathbb{E}(c-e-4)\boxtimes
\calo_{\mathbf{P}}|_{\Gamma\times Y}\right)|_{\{(S,y)\}}=H^2\bigl(E_y(c-e-4)|_S\bigr).
\end{equation}
From the triple (\ref{exact2}), it follows that
\begin{equation}\label{h2 le1}
h^2\bigl(E_y(c-e-4)|_S\bigr)\le1.
\end{equation}
On the other hand, the Grothendieck-Serre duality for a locally free $\calo_S$-sheaf $E_y|_S$ yields
\begin{equation}\label{h2=h0}
h^2\bigl(E_y(c-e-4)|_S\bigr)=h^0\bigl(E_y(-b)|_S\bigr).
\end{equation}

Next, the triple (\ref{exact2}) shows that 
\begin{equation}\label{P(y)}
\mathbf{P}(y):=\mathrm{Supp}(\coker(\psi\otimes
\mathbf{k}(y)))
\end{equation}
is a linear subspace of codimension at most $\mathbf{r}=\dim\mathbf{P}$ in $\mathbf{P}$. Hence this subspace $\mathbf{P}(y)$ is always nonempty, (\ref{coker psi})--(\ref{h2=h0}) give the following explicit description of $\mathbf{P}(y)$:
\begin{equation}\label{Py}
\mathbf{P}(y)=\{S\in\mathbf{P}|\ h^0(E_y(-b)|_S)=1\}.
\end{equation} 
Set $\tau(y)=\dim\mathbf{P}(y)$ and let

\begin{equation}\label{tau}
\tau:= \underset{y\in \mathcal{Y}}{\min}~\tau(y),\ \ \ 
Y:=\{y\in \mathcal{Y}|\ \tau(y)=\tau\}.
\end{equation}

Since $\mathcal{Y}$ is irreducible, the semicontinuity yields that $Y$ is a dense open subset of $\mathcal{Y}$. Moreover, from (\ref{varphi}) it follows that there exists a dense open subset $N$ of $N_s$, hence also of
$N_{\mathrm{nc}}$ and of $\overline{N}$: 
\begin{equation}\label{def of N}
\xymatrix{
N\ \ar@{^{(}->}[rrr]^-{\mathrm{dense\ open}} & & & N_s}
\end{equation}
defined by the fact that
\begin{equation}\label{N via Y}
Y=\varphi^{-1}\bigl(N\bigr)=\mathcal{Y}\times_{N_s}N\ \ \ \text{and}\ 
\ \ \varphi:Y\to N\ \text{is a principal}\ G\text{-bundle}.
\end{equation}  
The set $N$ is explicitly described as follows. For any point $[E]\in N_{\mathrm{nc}}$, consider the exact triple
\[
0\to E(b-c)\boxtimes\calo_{\mathbf{P}}(-1)\to E\boxtimes\calo_{\mathbf{P}}\to E\boxtimes\calo_{\mathbf{P}}|_{\Gamma}\to0
\]
and apply to it the functor $R^i\pr_{2*}$, where $\pr_2:\p3\times\mathbf{P}\to\mathbf{P}$ is the projection.
Then, similar to (\ref{exact2}), we obtain an exact triple
\begin{equation*}\label{exact3}
\mathbf{r}\calo_{\mathbf{P}}(-1)\overset{\psi_E}{\to}
\calo_{\mathbf{P}}\to\coker\psi_E\to0.
\end{equation*}
Similar to the above, set $\mathbf{P}\bigl([E]\bigr):=\mathrm{Supp}(\coker\psi_E)$, $\tau_E:=\dim\mathbf{P}([E])$. Then, as in (\ref{Py})--(\ref{tau}), we have 
\begin{equation}\label{PE}
\mathbf{P}\bigl([E]\bigr)=\{S\in\mathbf{P}|\ h^0(E(-b)|_S)=1\} \ \ \ \ \ \ 
\underset{[E]\in N_s}{\min}\tau_E=\tau,
\end{equation}
and
\begin{equation}\label{descriptn of N}
N=\{[E]\in N_s\ |\ \tau_E=\tau\}.
\end{equation}
Denote \begin{equation}\label{rmP([E])}
\mathrm{P}([E])=\mathrm{P}\cap\mathbf{P}\bigl([E]\bigr),
\end{equation} 
where $\mathrm{P}$ was defined in \eqref{rmP}. 
From \eqref{h0ES(-b)ne0}, \eqref{PE}, and the definition of $N_s$, it follows that $\mathrm{P}([E])$ is a nonempty, hence dense open subset of $\mathbf{P}\bigl([E]\bigr)$:
\begin{equation}\label{dense open P}
\xymatrix{
\mathrm{P}([E])\ \ar@{^{(}->}[rrr]^-{\mathrm{dense\ open}} & & & \mathbf{P}\bigl([E]\bigr)},\ \ \ \ \ \ \ [E]\in N_s.
\end{equation}

Now consider the subscheme $\boldsymbol{\mathcal{X}}$ of 
$\mathbf{P}\times \mathcal{Y}$ together with the projection $\boldsymbol{\theta}:\boldsymbol{\mathcal{X}}\to \mathcal{Y}$, defined as
\begin{equation}\label{bold cal X}
\boldsymbol{\mathcal{X}}:=\{\mathbf{x}=(S,y)\in \mathbf{P}\times \mathcal{Y}\ |\ S\in\mathbf{P}(y)\},\ \ \ \boldsymbol{\theta}:
\boldsymbol{\mathcal{X}}\to \mathcal{Y},\ (S,y)\mapsto y.
\end{equation}
Remark that, as $\rk L=1$ by (\ref{rk Li}), the triple (\ref{exact1}) 
twisted by $\calo_{\mathbf{P}}(1)\boxtimes L^{\vee}$
can be rewritten as
\begin{equation}\label{triple 33}
\calo_{\mathbf{P}}\boxtimes (L'\otimes L^\vee)
\overset{\psi}{\to}\calo_{\mathbf{P}}(1)\boxtimes 
\calo_{\mathcal{Y}}\to\mathbb{B}\to0,
\end{equation}
where
\begin{equation*}\label{line bdl G}
\mathbb{B}:=\calo_{\mathbf{P}}(1)\boxtimes \calo_{\mathcal{Y}}
|_{\boldsymbol{\mathcal{X}}}
\end{equation*}
is a line bundle on $\boldsymbol{\mathcal{X}}$. In view of (\ref{Py}) and (\ref{bold cal X}) the fibre of $\mathbb{B}$ over an arbitrary point $\mathbf{x}=(S,y)\in\boldsymbol{\mathcal{X}}$ has the description
\begin{equation}\label{fibre of G}
\mathbb{B}\otimes\mathbf{k}(\mathbf{x})=H^0(E_y(-b)|_S). 
\end{equation}
Applying to (\ref{triple 33}) the functor $p_{2*}$, where 
$p_2:\mathbf{P}\times \mathcal{Y}\to \mathcal{Y}$ is the projection, we obtain an exact triple
\begin{equation}\label{triple 35}
L'\otimes L^\vee\overset{f}{\to} S^{c-b}\mathcal{V}\otimes\calo_{\mathcal{Y}}
\to\mathbb{U}\to0,\ \ \ \mathbb{U}=p_{2*}\mathbb{B},
\end{equation}
where $\mathcal{V}=H^0(\op3(1))^{\vee}$, $f=p_{2*}\psi$ and 
\begin{equation}\label{bf scr X}
\boldsymbol{\mathcal{X}}=\mathbb{P}(\mathbb{U}).
\end{equation} 
In addition, 
$\mathbb{B}=\calo_{\mathbb{P}(\mathbb{U})}(1)$, and there is the canonical epimorphism
\begin{equation}\label{U onto G}
p_2^*\mathbb{U}\twoheadrightarrow\mathbb{B}.
\end{equation}
Remark that, since $\mathbb{E}$ has a natural $GL(N_m)$-linearization as a sheaf over $Q$, the sheaf 
$L'\otimes L^\vee$ has an induced $GL(N_m)$-linearization,
and the sheaf $S^{c-b}V\otimes\calo_{\mathcal{Y}}$ also has a (trivial) $GL(N_m)$-linearization. Hence by (\ref{triple 35}) the
sheaf $\mathbb{U}$ also inherits $GL(N_m)$-linearization. It follows that $\boldsymbol{\mathcal{X}}$ inherits $G$-action
such that $\boldsymbol{\theta}:\boldsymbol{\mathcal{X}}\to \mathcal{Y}$
is a $G$-equivariant morphism. Hence the geometric quotient 
\begin{equation*}\label{tilde X}
\mathcal{X}:=\boldsymbol{\mathcal{X}}//G 
\end{equation*}
is well-defined, and the canonical projection
\begin{equation*}\label{Phi}
\Phi:\ \boldsymbol{\mathcal{X}}\to \mathcal{X}
\end{equation*}
is a principal $G$-bundle. 

Furthermore, comparing (\ref{Py}) with (\ref{PE}), we see that, for any $[E]\in N_{\mathrm{nc}}$ and any $y\in\varphi^{-1}\bigl([E]\bigr)$ the fibre $\boldsymbol{\theta}^{-1}(y)=\mathbf{P}(y)$ as a subspace of 
$\mathbf{P}$ coincides with a subspace 
$\mathbf{P}\bigl([E]\bigr)$ of $\mathbf{P}$, and hence depends 
only on $[E]$. This implies that: (i) $\boldsymbol{\theta}$ is 
a $G$-equivariant morphism and therefore induces a morphism 
of categorical quotients $\theta_s: \mathcal{X}\to N_s$; 
(ii) a fibre $\theta_s^{-1}\bigl([E]\bigr)$ is a subspace 
$\mathbf{P}\bigl([E]\bigr)$ of $\mathbf{P}$. 
Thus $\theta_s:\mathcal{X}\to N_s$ is a $\mathbb{P}^{\tau}$-subfibration of the trivial fibration $\mathrm{P}\times N_s\to N_s$. 
Hence it is locally trivial.

Next, since $f=p_{2*}\psi$, we can rewrite (\ref{P(y)})
as
\begin{equation*}
\boldsymbol{\theta}^{-1}(y)=\mathbf{P}(y)=
P(\coker(\psi\otimes\mathbf{k}(y))),\ \ \ y\in \mathcal{Y}.
\end{equation*}
Set 
\begin{equation}\label{def X}
X:=\theta_s^{-1}(N)\cap P\times N,\ \ \ \ \ \ \theta:=\theta_s|_X: X\to N.
\end{equation}
By definition, $\theta: X\to N$ is a morphism with a fibre 
$\theta^{-1}([E])$ over an arbitrary point $[E]\in N$ 
being an open dense subset $\mathrm{P}([E])$ of subspace $\mathbf{P}([E])\simeq\mathbb{P}^{\tau}$ of $\mathrm{P}$
(see \eqref{dense open P}). Hence $\theta:X\to N$ is an open subfibration of the locally trivial $\mathbb{P}^{\tau}$-fibration $\theta_s: \mathcal{X}\to N_s$. 
Hence $\theta$ is also locally trivial. Furthermore, since $N$ is irreducible, it follows that $X$ is also irreducible.

We now arrive at the following result.

\begin{theorem}\label{Thm 2} 
(i) Let $X$ be defined in \eqref{def X}.
There is an open subfibration $\theta:\ X\to N$ of a locally trivial $\mathbb{P}^{\tau}$-fibration, and a fibre $\mathrm{P}([E])=\theta^{-1}([E])$ over an arbitrary point $[E]\in N$ is given by \eqref{rmP([E])}.
In other words, the set of closed points of the scheme $X$ is described as
\begin{equation}\label{tilde X tau}
X=\{(S,[E])\in \mathrm{P}\times N\ |\ h^0(E(-b)|_S)=1\}.
\end{equation}
In particular,
\begin{equation}\label{dim X}
\dim X=\dim \overline{N}+\tau,
\end{equation}
where $\dim\overline{N}$ is given by formula (\ref{dim N 1}).\\
(ii) Set $Y=\mathcal{Y}\times_{N_s}N$. There are cartesian diagrams
\begin{equation}\label{tilde Y,tilde X}
\xymatrix{ 
\boldsymbol{\mathcal{X}} \ar[d]_{\boldsymbol{\theta}}\ar[r]^{\Phi} & \mathcal{X} \ar[d]^{\theta_s} \\
\mathcal{Y}\ar[r]^-{\varphi} & N_s,}
\ \ \ \ \ \ \ \ \ 
\xymatrix{ 
\mathbf{X} \ar[d]_{\boldsymbol{\theta}}\ar[r]^{\Phi} & X \ar[d]^{\theta} \\ Y\ar[r]^-{\varphi} & N,}
\end{equation}
in which horizontal maps are principal $G$-bundles. Here the second diagram is obtained from the first via the commutative 
diagram
\begin{equation}\label{X in Xs}
\xymatrix{ 
\mathcal{X} \ar[d]_{\theta_s} & & X \ar[d]^{\theta} \ar@{_{(}->}[ll]_-{\mathrm{dense\ open}} \\ N_s & & N\ar@{_{(}->}[ll]_-{\mathrm{dense\ open}}.}
\end{equation}
Furthermore, vertical maps in the second diagram are open subfabrations of locally trivial $\mathbb{P}^{\tau}$-fibrations.\\
(iii) The composition $\mathbf{X}\overset{\mathrm{open}}{\hookrightarrow}\boldsymbol{\mathcal{X}}\overset{\boldsymbol{\theta}}{\to}\mathcal{Y} \hookrightarrow Q$ 
induces a family  
\begin{equation}\label{mathbf E}
\mathbf{E}=\mathbb{E}_{\mathbf{X}}
\end{equation} 
of generalized null correlation bundles on $\p3$ with base $\mathbf{X}$, where $\mathbb{E}$ is the universal quotient sheaf on $\p3\times Q$. \end{theorem}

\vspace{5mm}

\section{Family $\mathbf{E}$ of generalized null correlation bundles and related family of reflexive sheaves $\mathbf{F}$ on $\p3$}
\label{section 3}
 
\vspace{2mm}

In the first part of this section we study more closely generalized null correlation bundles $E$ of the family $\mathbf{E}$ introduced in Theorem \ref{Thm 2}(iii). In the second part we associate to $\mathbf{E}$ a family $\mathbf{F}$ of reflexive rank 2 sheaves on $\p3$. These two families will play the main role in subsequent constructions.

Consider an arbitrary sheaf $[E]\in N$. By definition (see (\ref{ein monad})--(\ref{calh})), the sheaf $E$ is the cohomology sheaf of the monad \eqref{lambda,mu} with the data \eqref{f1,...,f4}. From the definition of $N$ (see \eqref{def of Ms}--\eqref{def of Ns} and \eqref{tau}--\eqref{N via Y}), it follows that the monad  \eqref{lambda,mu} can be chosen in such a way that the related surface $S$ and the curve $C_0$ defined by \eqref{S,S'} and \eqref{C=S cap S'} are both smooth (hence irreducible).
In particular, $C_0$
is a smooth irreducible complete intersection curve with the conormal sheaf $N_{C_0/\p3}^{\vee}\simeq\calo_{C_0}(a-c)\oplus
\calo_{C_0}(b-c)$. 
Besides, (\ref{lambda,mu})--(\ref{C=S cap S'}) yield:
\begin{equation}\label{OS(C)}
\calo_S(C_0)\simeq\calo_{\p3}(S')|_S\simeq\calo_S(c-a).
\end{equation}
Furthermore, by \cite[Example 3.3]{Rao}, there is a well defined quotient sheaf $\calo_{C_0}(a+b-e)$ of $N_{C_0/\p3}^{\vee}$, 
\begin{equation}\label{quotient of N}
N_{C_0/\p3}^{\vee}=\calo_{C_0}(a-c)\oplus\calo_{C_0}(b-c)
\twoheadrightarrow\calo_{C_0}(a+b-e),
\end{equation} 
which determines a double scheme structure $\overline{C}_0$ on $C_0$ with the following properties: 

(i) the curve $\overline{C}_0$ is a locally complete intersection curve satisfying the exact triple
\begin{equation}\label{double str on C}
0\to\calo_{C_0}(a+b-e)\to\calo_{\overline{C}_0}\to\calo_{C_0}
\to0;
\end{equation}

(ii) $\overline{C}_0$ is the zero-scheme of some section of the sheaf $E(c-a-b)$:
\begin{equation}\label{zero-set}
\overline{C}_0=(s)_0,\ \ \ 0\ne s\in H^0(E (c-a-b)).
\end{equation}
Remark that (\ref{zero-set}) implies an exact triple
\begin{equation}\label{s, alpha}
0\to\op3(a+b-c)\overset{s}\to E\overset{\alpha}{\to}
\cali_{\overline{C}_0}(c-a-b+e)\to0.
\end{equation}
Note first that, since $c-a-e>0$, it follows that, in 
(\ref{quotient of N}), the quotient sheaf $\calo_{C_0}(a+b-e)$ 
does not coincide with the direct summand $\calo_{C_0}(b-c)$ 
of the conormal sheaf $N_{C_0/\p3}^{\vee}$, so that the curve 
$\overline{C}_0$ defined in (\ref{quotient of 
N})--(\ref{zero-set})
is not a subscheme of the surface $S$. Therefore, the sheaf 
$\kappa=\ker(\calo_{C'_0}\twoheadrightarrow\calo_{C_0})$,
where the scheme $C'_0$ is defined as the scheme-theoretic intersection $C'_0=\overline{C}_0\cap S$, has dimension at most zero:
\begin{equation}\label{dim kappa le0}
0\to\kappa\to\calo_{C'_0}\to\calo_{C_0}\to0,\ \ \ \ \ \dim\kappa\le0.
\end{equation}
(Here, the inequality $\dim\kappa\le0$ is provided by  smoothness and irreducibility of the curve $C_0$.)
This together with (\ref{OS(C)}) implies an exact triple
\begin{equation}\label{triple IZS}
0\to\cali_{Z,S}(e-b)\to\calo_S(c-a-b+e)\to
\calo_{C'_0}(c-a-b+e)\to0
\end{equation}
and a relation $\kappa\simeq\calo_Z$ for some subscheme $Z$ of $S$ of
dimension at most zero:
\begin{equation}\label{dim Z le0}
\dim Z\le0.
\end{equation}
The exact triples
$$
0\to\calo_{\p3}(e-a)\overset{\cdot S}{\to}\calo_{\p3}(c-a-b+e)\to
\calo_S(c-a-b+e)\to0,
$$
$$
0\to\cali_{\overline{C}_0}(c-a-b+e)\to\calo_{\p3}(c-a-b+e)\to\calo_{\overline{C}_0}(c-a-b+e)\to0
$$
together with (\ref{triple IZS}) extend to a commutative diagram 
\begin{equation}\label{diag with C_0}
\xymatrix{
& 0 \ar[d] & 0 \ar[d] & 0 \ar[d] & \\
0 \ar[r] & \cali_{C_0}(e-a) \ar[d]\ar[r] & \cali_{\overline{C}_0}(c-a-b+e) \ar[d] \ar[r]^{\beta} & \cali_{Z,S}(e-b) \ar[d] \ar[r] & 0 \\
0 \ar[r] & \calo_{\p3}(e-a) \ar[d] \ar[r]^-{\cdot S} &
\op3(c-a-b+e) \ar[d] \ar[r] &
\calo_S(c-a-b+e) \ar[d] \ar[r] & 0 \\
0\ar[r] & \calo_{C_0}(e-a) \ar[d] \ar[r] & \calo_{\overline{C}_0}(c-a-b+e) \ar[d] \ar[r] & \calo_{C'_0}(c-a-b+e) \ar[r] \ar[d] & 0 \\
& 0 & 0 & 0. &} 
\end{equation}
Now the composition of morphisms $\beta\circ\alpha$, where 
$\alpha$ is taken from (\ref{s, alpha}) and $\beta$ is defined
in this diagram,
decomposes as
\begin{equation}\label{epsilon}
\beta\circ\alpha:\ E\overset{\otimes\calo_S}{\twoheadrightarrow}E|_S
\overset{\gamma}{\twoheadrightarrow}\cali_{Z,S}(e-b)
\end{equation}
for some epimorphism $\gamma:\ E|_S
\overset{\gamma}{\twoheadrightarrow}\cali_{Z,S}(e-b)$. 

Note that (\ref{dim Z le0}) implies the equalities
$\cale xt^i_{\op3}(\calo_Z,\op3)=0,\ i=1,2,$ which together with the exact sequence 
$\cale xt^1_{\op3}(\calo_Z,\op3)\to\cale xt^1_{\op3}
(\calo_Z,\calo_S)\to$\\ 
$\cale xt^2_{\op3}(\calo_Z,\op3(b-c))$
obtained from the exact triple 
$0\to\op3(b-c)\overset{\cdot S}{\to}\op3\to\calo_S\to0$
yield
\begin{equation}\label{vanish Ext1}
\cale xt^1_{\calo_S}(\calo_Z,\calo_S)=
\cale xt^1_{\op3}(\calo_Z,\calo_S)=0.
\end{equation}
Applying the functor
$\cale xt^{\cdot}_{\calo_S}(-,\calo_S)$ to the exact triple
$0\to\cali_{Z,S}\to\calo_S\to\calo_Z\to0$ and using (\ref{vanish Ext1}) we obtain
\begin{equation}\label{dual to I}
\cali_{Z,S}^{\vee}\simeq\calo_S.
\end{equation}
Dualizing the morphism $\gamma$ in (\ref{epsilon}) and using
(\ref{dual to I}) and the isomorphism $(E|_S)^{\vee}\simeq (E|_S)(-e)$, after twisting it by $\calo_S(e-b)$ we obtain a morphism
$\mathbf{s}=(\gamma)^{\vee}(e-b):\ \calo_S\to E(-b)|_S$,
i.e. a section $0\ne \mathbf{s}\in H^0(E(-b)|_S)$.  This section is a subbundle morphism on $S\setminus Z$, hence in view of (\ref{dim Z le0}) it extends to the Koszul exact triple
\begin{equation*}\label{Koszul}
0\to\calo_S\overset{\mathbf{s}}{\to}E(-b)|_S\xrightarrow
{\mathbf{s}^{\vee}\otimes\wedge^2\mathbf{s}}
\cali_{Z,S}(e-2b)\to0.
\end{equation*}
This triple shows that 
\begin{equation}\label{Z=(s)0}
\gamma=\mathbf{s}^{\vee}\ \ \  and\ \ \  Z=(\mathbf{s})_0.
\end{equation} 
A standard computation using \eqref{dim Z le0} and \eqref{Z=(s)0} shows that 
\begin{equation}\label{l(Z)}
l(Z)=c_2(E(-b)|_S)=(c-a)(c-b)(c+a-e)>0,
\end{equation}
hence \eqref{dim Z le0} implies that
\begin{equation}\label{dim Z=...=0}
\dim Z=\dim(\mathbf{s})_0=0.
\end{equation}
Besides, the equality 
\begin{equation}\label{again h0=1}
h^0(E(-b)|_S)=1
\end{equation}
follows from (\ref{h2 le1}) and (\ref{h2=h0}) (or, equivalently, from \eqref{tilde X tau}, since by assumption $(S,[E])\in X$). Hence, $H^0(E(-b)|_S)$ is spanned by $\mathbf{s}$. 

From \eqref{dim Z=...=0}--\eqref{again h0=1}, it follows
\begin{theorem}\label{Thm 3.1}
For any point $(S,[E])\in X$, one has $h^0(E(-b)|_S)=1$ and $\dim(s)_0=0$ for any $0\ne s\in H^0(E(-b)|_S)$.
\end{theorem}

Consider the incidence variety $\Gamma$ introduced in \eqref{Gamma}. Using the embedding $\mathbf{X}\hookrightarrow P\times Y$ 
(cf. Theorem \ref{Thm 2}), set
$$
\mathbf{\Gamma}=(\Gamma\times Y)\times_{\mathbf{P} \times Y}\mathbf{X}
$$
and let $\rho:\mathbf{\Gamma}\to \mathbf{X}$ be the natural projection. Set
\begin{equation*}\label{E tau}
\mathbf{L}=L_{\mathbf{X}}, 
\end{equation*}
where the invertible $\calo_{\mathcal{Y}}$-sheaf $L$ was defined in (\ref{L,L'}).
Consider the family of generalized null correlation bundles $\mathbf{E}$ defined in
Theorem \ref{Thm 2}(iii). The first equality in (\ref{L,L'}) and the base change imply 
$$
\mathbf{L}=R^2\rho_*(\mathbf{E}
|_{\mathbf{\Gamma}}(c-e-4))
$$
(here and below we use the convention \eqref{Amn} on notation), so that the relative Serre duality for the projection
$\rho$ yields
\begin{equation}\label{n1}
\mathbf{L}^\vee\simeq
\rho_*(\mathbf{E}|_{\mathbf{\Gamma}}(-b)). 
\end{equation}
Respectively, for an arbitrary point $\mathbf{x}\in\mathbf{X}$ and a surface $S_{\mathbf{x}}:=\mathbf{\Gamma}\times_{\mathbf{X}}
\{\mathbf{x}\}\subset\mathbb{P}^3$, we have
\begin{equation}\label{n2}
\mathbf{L}^\vee\otimes_{\calo_{\mathbf{X}}}\mathbf{k}
(\mathbf{x})=H^0(\mathbf{E}(-b)|_{S_{\mathbf{x}}}).
\end{equation}
Following our convention on notation, denote 
$\mathbf{L}_{\mathbf{\Gamma}}=(\mathrm{id}_{\p3}\times\rho)^*
\mathbf{L}$, $\mathbf{E}_{\mathbf{\Gamma}}=
(\mathrm{id}_{\p3}\times\rho)^*
\mathbf{E}$. The isomorphism (\ref{n1}) induces a section $s_{\mathbf{\Gamma}}\in H^0(\mathbf{E}_{\mathbf{\Gamma}}(-b)
\otimes\mathbf{L}_{\mathbf{\Gamma}})$ defined as
\begin{equation}\label{n4}
s_{\mathbf{\Gamma}}:\ \calo_{\mathbf{\Gamma}}=
\rho^*\rho_*(\mathbf{E}_{\mathbf{\Gamma}}
(-b))\otimes\mathbf{L}_{\mathbf{\Gamma}}\overset{ev}{\to}
\mathbf{E}_{\mathbf{\Gamma}}(-b)\otimes
\mathbf{L}_{\mathbf{\Gamma}}.
\end{equation}
Let
$$
\mathcal{Z}=(s_{\mathbf{\Gamma}})_0
$$
be the zero scheme of this section. By the base change for any
$\mathbf{x}\in\mathbf{X}$ the scheme 
\begin{equation}\label{n3}
Z_{\mathbf{x}}=\mathcal{Z}\cap S_{\mathbf{x}}
\end{equation} 
is the zero set of the section $s_{\mathbf{\Gamma}}|_{S_{\mathbf{x}}}\in H^0(\mathbf{E}(-b)|_{S_{\mathbf{x}}})$,
hence from Theorem \ref{Thm 3.1} and the definition of $\mathbf{X}$ we have
$2=\codim_{S_{\mathbf{x}}}Z_{\mathbf{x}}=
\codim_{\mathbf{\Gamma}}\mathcal{Z}$,
so that
\begin{equation}\label{codim}
\codim_{\mathbb{P}^3\times\mathbf{X}}\mathcal{Z}=
\codim_{\mathbb{P}^3\times\{\mathbf{x}\}}Z_{\mathbf{x}}=3.
\end{equation}
Use (\ref{n4}) and the relation 
\begin{equation}\label{Edual}
\mathbf{E}^\vee\simeq\mathbf{E}(-e),
\end{equation}
and consider the composition 
$\boldsymbol{\varepsilon}:\ \mathbf{E}\overset{\otimes\calo_{\mathbf{\Gamma}}}
{\twoheadrightarrow}\mathbf{E}|_{\mathbf{\Gamma}}
\overset{s_{\mathbf{\Gamma}}^\vee}{\twoheadrightarrow}
\mathcal{I}_{\mathcal{Z},\mathbf{\Gamma}}(e-b)\otimes
\mathbf{L}_{\mathbf{\Gamma}}.$
Setting 
\begin{equation}\label{def of bfF}
\mathbf{F}:=\ker\boldsymbol{\varepsilon},
\end{equation}
we obtain an exact triple
\begin{equation}\label{triple E,F}
0\to\mathbf{F}\to\mathbf{E}
\overset{\boldsymbol{\varepsilon}}{\to}
\mathcal{I}_{\mathcal{Z},\mathbf{\Gamma}}(e-b)\otimes
\mathbf{L}_{\mathbf{\Gamma}}\to0.
\end{equation}
\begin{remark}\label{reduction step}
(i) Take any point $\mathbf{x}\in\mathbf{X}$ and restrict 
the last triple onto $\mathbb{P}^3\times\{\mathbf{x}\}$. We will obtain the triple 
\begin{equation}\label{F from E}
0\to F\to E\overset{\varepsilon}{\to}\cali_{Z,S}(e-b)\to0, 
\end{equation}
where $E=\mathbf{E}|_{\mathbb{P}^3\times\{\mathbf{x}\}}$ is a generalized null correlation bundle, $S=S_{\mathbf{x}}$ is a smooth surface from the linear series $\mathbf{P}$ defined by the point $\mathbf{x}$ (namely, $(S,[E])=\Phi(\mathbf{x})$), $Z=Z_{\mathbf{x}}$, $\cali_{Z,S}$ is the ideal sheaf of $Z$ in $S$, and $F=\ker\varepsilon$. This triple is an analogue of the so-called reduction step in the sense of Hartshorne \cite[Prop. 9.1]{H-r}, hence $F$, and therefore also $\mathbf{F}$, is a reflexive sheaf.\\
(ii) In \eqref{F from E} $Z$ is the zero-set of the section $s=\varepsilon^{\vee}$ of the bundle $E(-b)|_S$. Therefore a standard computation using \eqref{l(Z)} and the relations $c_1(E)=e,\ c_2(E)=c^2-a^2-b^2- e(a-b-c)$, $\deg S=c-b$ shows that 
\begin{equation}\label{ci(F) 0}
\begin{split}
& c_1(F)=e+b-c,\\
& c_2(F)=c^2-a^2-bc-e(c-a-b),\\
& c_3(F)=l(Z)=c_2(E(-b)|_S)=(c-a)(c-b)(c+a-e).
\end{split}
\end{equation}
Since the sheaf $F$ is determined uniquely up to an isomorphism by the pair $x=([E],S)\in X$ as the kernel of
an epimorphism $\varepsilon$ in \eqref{F from E}, we will
also use the following  notation for $F$:
\begin{equation}\label{F(x)}
F=F(x)=F(E,S).
\end{equation}
(iii) From \eqref{triple E,F} it follows that the sheaf $\mathbf{F}$ is determined by the sheaf $\mathbf{E}$ uniquely
up to an isomorphism. Hence, since $\mathbf{E}$ inherits a $GL(N_m)$-linearization as a quotient sheaf over (an open subset of) the Quot-scheme, the sheaf $\mathbf{F}$ also inherits a $GL(N_m)$-linearization.
\end{remark}
Since by construction 
\begin{equation}\label{F for w}
F=\mathbf{F}|_{\mathbb{P}^3\times\{\mathbf{x}\}},\ \ \ 
\mathbf{x}\in\mathbf{X},
\end{equation} 
and $\det\mathbf{E}\simeq\calo_{\p3\times \mathbf{X}}(e)$, it follows from (\ref{ci(F) 0}) that
\begin{equation}\label{det F}
\det\mathbf{F}\simeq\calo_{\p3\times\mathbf{X}}
(e+b-c).
\end{equation}
As $\mathbf{F}$ is a rank 2 reflexive sheaf on $\p3\times\mathbf{X}$ by Remark \ref{reduction step}(i),  (\ref{det F}) implies
\begin{equation}\label{F dual}
\mathbf{F}^{\vee}=\mathbf{F}(c-e-b).
\end{equation}

Next, from (\ref{exact0}) follows the relation $N_{\mathbf{\Gamma}/\p3\times\mathbf{X}}
\simeq\calo_{\mathbf{\Gamma}}(c-b,1)$, and (\ref{codim}) implies
$$
\mathcal{E}xt^1(\mathcal{I}_{\mathcal{Z},\mathbf{\Gamma}}
(e-b)\otimes\mathbf{L}_{\mathbf{\Gamma}},\calo_{\p3\times\mathbf{X}})=\mathcal{E}xt^1(\mathbf{L}_{\mathbf{\Gamma}}(e-b),\calo_{\p3\times\mathbf{X}})=
\mathbf{L}_{\mathbf{\Gamma}}^{\vee}(c-e,1).
$$
Thus, dualizing the triple (\ref{triple E,F}) and using 
(\ref{Edual}) and (\ref{F dual}) we obtain an exact triple
\begin{equation}\label{E from F}
0\to\mathbf{E}(b-c)\to\mathbf{F}\overset{\boldsymbol{\psi}}{\longrightarrow}\mathbf{L}_{\mathbf{\Gamma}}^{\vee}(b,1)\to0.
\end{equation}
Note that the restriction of (\ref{E from F}) onto $\p3\times\{\mathbf{x}\}$
for any $\mathbf{x}\in\mathbf{X}$ yields an exact triple
\begin{equation}\label{E from F restricted}
0\to E(b-c)\to F\overset{\psi}{\to}
\mathcal{O}_S(b)\to0, \ \ \ E=\mathbf{E}|_{\p3\times\{\mathbf{x}\}}, \ \ \   F=\mathbf{F}|_{\p3\times\{\mathbf{x}\}}.
\end{equation}

\vspace{5mm}

\section{Plan of the proof of the main result}
\label{section 4.0}

\vspace{2mm}

In this section we outline a general plan of the proof of the main result of the paper - Theorem \ref{main Thm}. It consists of four steps.

\vspace{2mm}
\textbf{Step 1.} This step is described in detail in Section \ref{section 4}. We consider the set
\begin{equation}\label{W as a set}
\begin{split}
& W=\{([E],S,C)\ |\ ([E],S)\in X,\ C=(s)_0\ is\ a\ smooth\ curve,\\
& where\ 0\ne s\in H^0(F(E,S)(c-a-b))\}.
\end{split}
\end{equation}
(Remind that here we use the notation $F(E,S)$ introduced in \eqref{F(x)} for a reflexive sheaf $F$ determined by the point $x=([E],S)\in X$ - see Remark \ref{reduction step}(i-ii).)
It is proved in Corollary \ref{cor 5.3} that this set 
underlies a variety $W$ with a projection $\pi:W\to X$ which is an open subfibration of a locally trivial $\mathbb{P}^{m}$-fibration over $X$, where $m$ is 
given by \eqref{formula for m}. 
We thus have a diagram of cartesian squares extending
the right diagram \eqref{tilde Y,tilde X}:
\begin{equation}\label{diag for X,W}
\xymatrix{ 
**[l]\mathbf{E}_{\mathbf{W}}, \mathbf{F}_{\mathbf{W}}\ \   \mathbf{W}\ar[d]_{\boldsymbol{\pi}}\ar[r]^{\tilde{\Phi}} & **[r]W \ar[d]^{\pi} \ni ([E,]S,C)    \\
**[l]\mathbf{E}, \mathbf{F}\ \ \
\mathbf{X} \ar[d]_{\boldsymbol{\theta}}\ar[r]^{\Phi} & X \ar[d]^{\theta} \\
Y\ar[r]^-{\varphi} & N,}
\end{equation}
in which horizontal maps are principal $G$-bundles.
Here $\mathbf{E}$ and $\mathbf{F}$ are the families of $\op3$-sheaves with base $\mathbf{X}$ introduced in \eqref{mathbf E} and \eqref{def of bfF}, and 
$\mathbf{E}_{\mathbf{W}}$ and $\mathbf{F}_{\mathbf{W}}$
are their lifts onto $\p3\times\mathbf{W}$.

\vspace{2mm}
\textbf{Step 2.} At this step, which is performed in detail in Section \ref{section 5}, we construct a new family $\underline{\mathbf{F}}$ of reflexive sheaves 
on $\p3$, of the type described in Remark \ref{reduction step}, and with a rational base $\mathbf{T}$. These data $\mathbf{T}$ and $\underline{\mathbf{F}}$ are explicitely described in
\eqref{bfT as a set} and \eqref{univ FT} below. We then
restrict our consideration to a certain dense open
subset $T$ of $\mathbf{T}$ which will be essential
for our subsequent arguments.

We start with the linear series $\mathbf{P}=|\calo_{\p3}(c-b)|$ introduced in \eqref{bfP} and consider its dense open subset $\mathrm{P}$
of smooth surfaces - see \eqref{rmP}. Set
\begin{equation}\label{bfR=}
\mathrm{R}:=\{(S,C)\in\mathrm{P}\times\mathrm{Hilb}_{\p3}\ |\ C\in|\calo_S(c-a)|\ is\ a\ smooth\ curve\},
\end{equation}
together with a natural projection
$\mathrm{r}:\mathrm{R}\to\mathrm{P},\ (S,C)\mapsto S$.
\begin{remark}\label{fibre of r}
Since any $S\in\mathrm{P}$ is a smooth (hence irreducible) surface, it follows from the cohomology of the exact triple
$0\to\op3(b-a)\to\op3(c-a)\to\calo_S(c-a)\to0$
that\\ 
(i) the fibre $\mathrm{r}^{-1}(S)$ is an open dense subset of the linear series $|\calo_S(c-a)|$ consisting of smooth curves
and
$$
\dim \mathrm{r}^{-1}(S)=\binom{c-a+3}{3}-
\binom{b-a+3}{3}-1,
$$
and all the curves of this linear series are complete intersections of the form 
\begin{equation}\label{complete int}
C=S\cap S',\ \ \ S'\in|\op3(c-a)|;
\end{equation}
(ii) the projection $\mathrm{r}:\mathrm{R}\to\mathrm{P}$
is an open subfibration of a locally trivial projective fibration with fibre $|\calo_S(c-a)|$ over a point $S\in \mathrm{P}$; hence, since $\mathrm{P}$ is rational, $\mathrm{R}$ is rational as well; 
moreover,
\begin{equation}\label{dim R}
\dim\mathrm{R}=\dim\mathrm{P}+\dim \mathrm{r}^{-1}(S)=\binom{c-b+3}{3}+
\binom{c-a+3}{3}-\binom{b-a+3}{3}-2.
\end{equation}
\end{remark}

Take an arbiitrary point $(S,C)\in\mathrm{R}$ and consider the group
\begin{equation}\label{Ext1 group}
\mathrm{Ext}^i(x):=\mathrm{Ext}^i(\cali_C(c-2a-b+e),
\calo_{\p3}).
\end{equation}
In Section \ref{section 5} we prove that the dimension of this group does not depend on the point $x$ and is given by
the formula \eqref{dim Ext1}. This implies that the set
\begin{equation}\label{bfT as a set}
\mathbf{T}=\{t=(x,\xi)\ |\ x=(S,C)\in\mathrm{R},\ \xi\in P(\mathrm{Ext}^1(x))\}
\end{equation}
is the set of closed points of the variety (denoted below by the same letter $\mathbf{T}$) of dimension given by the formula \eqref{dim T}, and the projection
\begin{equation}\label{set-th def of mu}
\boldsymbol{\mu}:\ \mathbf{T}\to \mathrm{R},\ \ \ (x,\xi)\mapsto x
\end{equation}
is a locally trivial projective fibration. In particular,
since $\mathrm{R}$ is rational, $\mathbf{T}$ is also rational.

Furthermore, in Theorem \ref{Thm 7} we state that on
$\p3\times \mathbf{T}$ there is a sheaf $\underline{\mathbf{F}}$ defined as the universal extension sheaf
\begin{equation}\label{univ FT}
0\to\calo_{\p3}(a+b-c)\boxtimes\calo_{\mathbf{T}}(1)\to
\underline{\mathbf{F}}\to\cali_{\boldsymbol{\Sigma},\p3\times
\mathbf{T}}(e-a),\to0,\ \ \ \boldsymbol{\Sigma}=\Sigma
\times_{\mathrm{R}}\mathbf{T},
\end{equation}
where $\Sigma$ is the incidence subvariety of $\p3\times\mathrm{R}$ defined as
\begin{equation}\label{def of Sigma}
\Sigma:=\{(x,S,C)\in\p3\times\mathrm{R}\ |\ x\in C\}.
\end{equation}
Here $\underline{\mathbf{F}}$ may be considered as a family of reflexive $\op3$-sheaves with base $\mathbf{T}$ and with Chern classes given by (\ref{ci(F) 0}) - see Remark \ref{Serre constr}.

In the last part of Section \ref{section 5} we prove one 
technical result about reflexive sheaves of the family 
$\underline{\mathbf{F}}$ with base $\mathbf{T}$.
It shows that, if for $t=(S,C,\xi)$, a sheaf $F_t$ of the 
family $\underline{\mathbf{F}}$ has an epimorphism onto an 
invertible $\calo_S$-sheaf $\calo_S(b)$, then the kernel of 
this morphism is a generalized null correlation bundle twisted 
by $\calo_{\p3}(b-c)$, just as in the exact triple \eqref{E 
from F restricted} in which we put $F=F_t$. It is proved in 
Theorem \ref{Thm 10}. A principal technical point used in the 
proof is the following specific property of any generalized 
null correlation bundle $E$: it has the cohomology 
$H^0_*(\calo_{\p3})$-module $H^1_*(E)$ with one generator. 
This Theorem is crucial for further constructions.

\vspace{2mm}
\textbf{Step 3.} At this step which is worked out in detail in 
Section \ref{section 6}, we use the above family of reflexive 
sheaves $\underline{\mathbf{F}}$ to construct a family 
$\underline{\mathbf{E}}$ of generalized null correlation 
bundles with rational base $V$. For this, we first construct
a locally trivial projective bundle $\boldsymbol{\lambda}:\ 
\mathbf{U}\to\mathbf{T}$ with fibre 
$\boldsymbol{\lambda}^{-1}(t)$ over an arbitrary point 
$t=(S,C,\xi)\in T$ equal to the projectivized vector space 
$\mathrm{Hom}(F_t,\calo_S(b))$. The local triviality of this 
projective fibration is a consequence of Theorem \eqref{Thm 
11} which, in particular, states that the dimension of the 
above space $\mathrm{Hom}(F_t,\calo_S(b))$ does not depend on 
$t$. As a corollary of Theorems \ref{Thm 10} and \ref{Thm 11}
we then find dense open subsets $T$ of $\mathbf{T}$ and 
$V$ of $\mathbf{U}$ such that (i) $\lambda=\boldsymbol{\lambda}
|_V: V\to T$ is a surjection and, for $(t,\mathbf{k}\psi)\in V$, and (ii) the morphism $\psi:F_t\to\calo_S(b)$ is surjective. More precisely, $V$ is set-theoretically defined as the set of data
\begin{equation}
\begin{split}
& V=\{(S,C,\xi,\mathbf{k}\psi)\ |\ (S,C,\xi,\mathbf{k}\psi) \ satisfies\ the\ above\ conditions\ (i)-(ii)\\
& and\ the\ open\ condition\ [\ker(\psi)(c-b)]\in N \}
\end{split}
\end{equation} 
(the precice definition of $V$ is given in \eqref{underline E}).
As a consequence, we obtain a family $\underline{\mathbf{E}}$ of generalized null correlation bundles related to the family
$\underline{\mathbf{F}}$ via the exact triple 
\begin{equation}\label{tilde E from tilde F}
0\to\underline{\mathbf{E}}(b-c)\to\underline{\mathbf{F}}_V
\overset{\Psi}{\to}\calo_{\Gamma_{V}}(b)
\otimes\calo_{\p3}\boxtimes
\calo_{V}(1)\to0
\end{equation}
(see Remark \ref{global triple E,F}). Here $\Gamma_V\subset\p3\times V$ is the graph of the family of surfaces $S$ and $\calo_V(1)$ is the restricted onto $V$ Grothendieck sheaf of the above mentioned projective fibration - see \eqref{notation for B}. This triple is the relativized over $V$ version of the exact triple \eqref{E from F restricted}.
As a result of the constructions of Steps 2 and 3, we obtain 
the following diagram of morphisms:
\begin{equation}\label{diag for T,B}
\xymatrix{ 
\underline{\mathbf{F}}_V,\ \underline{\mathbf{E}} & V\ \ar[d]_-{\lambda}\ar@{^{(}->}[r]^{\mathrm{open}} & \mathbf{U} \ar[d]_-{\boldsymbol{\lambda}} & \\
& T\ \ar[d]_{\mu}\ar@{^{(}->}[r]^{\mathrm{open}} & 
\mathbf{T}  \ar[dl]_-{\boldsymbol{\mu}}  &  \boldsymbol{\Sigma}\ar[l]\ar[dl]\\
& \mathrm{R}\ar[d]_-{\mathrm{r}} & \Sigma\ar[l] & \\ 
& {}\ {\mathrm{P}}, & & } 
\end{equation}
together with the family of $\op3$-sheaves $\underline{\mathbf{F}}$ with base $\mathbf{T}$ and the induced families of $\op3$-sheaves $\underline{\mathbf{F}}_V,\ \underline{\mathbf{E}}$ with base $V$.
Remind that, in this diagram, varieties $R,\ T,\ \Sigma$ and $\boldsymbol{\Sigma}$ were defined in \eqref{bfR=}, \eqref{bfT as a set}, \eqref{def of Sigma} and \eqref{univ FT}, respectively.  
 
\vspace{2mm}
\textbf{Step 4.}
At this final step performed in Section \ref{section 7} we show that there is an isomorphism
$f:\ W\overset{\simeq}{\to}V.$ Set-theoretically the map
$f:W\to V$ on closed points is given as follows.

For any $([E],S,C)\in W$ consider an exact triple \eqref{F from E}. Dualizing it we obtain a) an exact triple \eqref{E from F restricted} with $F=F(E,S)$ and an epimorphism  $\psi:F\to\calo_S(b)$, and b) an extension class $\xi\in
P(\Ext^1(\cali_C(c-2a-b+e),\op3))$ given by an exact triple \eqref{s for F} with $F=F(E,S)$. Then we define $f$ as:
\begin{equation}\label{set-theoretic def of f}
f([E],S,C):=(S,C,\xi,\mathbf{k}\psi).
\end{equation}
From the description of $V$ given in Step 3 it follows that 
the point $f([E],S,C)$ belongs to $V$. 

Respectively, the inverse $h=f^{-1}: V\to W$ of $f$ is set-theoretically described as:
\begin{equation}\label{set-theoretic def of h}
h(S,C,\xi,\mathbf{k}\psi):=([E],S,C),\ \ \ where\ \ \  E=
\ker(\psi)(c-b).
\end{equation}

In Theorem \ref{main Thm}(i) we prove that the map $f$, respectively, its inverse $h$ is the 
underlying map of an isomorphism between $W$ and $V$. The idea
is to relate the diagrams \eqref{diag for X,W} and \eqref{diag for T,B}. We first construct a $G$-invariant morphism
$\mathbf{f}_V:\ \mathbf{W}\to V$ which descends to the desired morphism $f:W\to V$ satisfying the relation $\mathbf{f}_V=f\circ\tilde{\Phi}$ since $\tilde{\Phi}:\mathbf{W}\to W$ is categorical quotient.

Next, we construct a principal $PGL(N_m)$-bundle  
$\mathbf{\Phi}:\ \mathbf{V}\to V$ and a morphism
$\mathbf{f}:\mathbf{W}\to\mathbf{V}$ making the diagram
\begin{equation}
\xymatrix{\mathbf{W} \ar[d]_{\tilde{\Phi}}\ar[r]^{\mathbf{f}}
& \mathbf{V}\ar[d]^{\boldsymbol{\Phi}} \\
W\ar[r]^-{f} & V}
\end{equation}
commutative (see \eqref{var bf V}-\eqref{bf Phi} and \eqref{mor bf f}-\eqref{diag f,fB} for details).

Last, we construct the morphisms $\mathbf{h}:\mathbf{V}\to\mathbf{W}$ and $h:V\to W$
making the diagram
\begin{equation*}
\xymatrix{\mathbf{W}\ar[d]_{\Phi} & \mathbf{V} \ar[l]_-{\mathbf{h}} \ar[d]^{\mathbf{\Phi}} \\
W & V\ar[l]_-{h}}
\end{equation*}
commutative and show that $\mathbf{h}$ and $h$ are inverse, 
respectively to $\mathbf{f}$ and $f$ (see \eqref{morphism 
h}-\eqref{diag bf h, h} for details).

Technical aspects of the proof are based on the universal 
properties of Quot-schemes, Hilbert schemes and projectivized 
spaces of extensions involved in the constructions of the 
families $\mathbf{E}_{\mathbf{W}}, \mathbf{F}_{\mathbf{W}}$ in 
diagram \eqref{diag for X,W} and the families 
$\underline{\mathbf{F}}_V,\ \underline{\mathbf{E}}$ in diagram 
\eqref{diag for T,B}.

In Theorem \ref{main Thm}(ii)--(iii) we obtain the main result of the paper, the stable rationality of the space $\overline{N}(e,a,b,c)$ and, respectively, its rationality for
$(e,a)\ne(0,0),\ c>2a+b-e$, and $b>a$, as a quick consequence
of the statement (i) of this Theorem.

\vspace{5mm}

\section{Properties of reflexive sheaves of the family $\mathbf{F}$}
\label{section 4}

\vspace{2mm}

In this section we study more closely reflexive sheaves $F$ of the family $\mathbf{F}$ -- see (\ref{F for w}). 
Note that an arbitrary sheaf $F$ is obtained from a generalized null correlation bundle $[E]\in N$ by the triple (\ref{F from E}).

This consideration leads to the following theorem. 

\begin{theorem}\label{Thm 3}
For any $[E]\in N$ the following statements hold.\\
(i) There exists a surface $S\in\theta^{-1}\bigl([E]\bigr)$ such that the reflexive sheaf $F$ defined by 
the pair $\bigl(S,[E]\bigr)$ as in Remark  \ref{reduction step} satisfies the conditions
\begin{equation}\label{h0(F...)=1}
h^0(F(c-a-b))=\Biggl\{
\begin{array}{ccc}
1, &\mathrm{if}\ (e,a)\ne(0,0), \\
2, &\ \ \ \ \mathrm{if}\ e=a=0,\ b>0, \\
3, &\ \mathrm{if}\ e=a=b=0,
\end{array}
\end{equation}
\begin{equation}\label{h1(F...)=0}
h^1(F(c-a-b))=0.
\end{equation}
(ii) For any $0\ne s\in H^0(F(c-a-b))$ there an exact triple 
\begin{equation}\label{s for F}
0\to\op3\overset{s}{\to}F(c-a-b)\to\cali_C(c-2a-b+e)\to0,
\end{equation}
where $C=(s)_0$ is a complete intersection curve $C=S\cap S'$, where $S'$ is certain surface of degree $c-a$ in $\mathbb{P}^3$. In addition, $C=(s)_0$ is smooth for a general
$s\in H^0(F(c-a-b))$.
\\
(iii) In case $h^0(F(c-a-b))\le2$, the space 
$P(H^0(F(c-a-b)))$ is naturally identified with a linear 
subspace of the linear series $|\calo_S(C)|=|\calo_S(c-a)|$. 
In case  $h^0(F(c-a-b))=3$, the space 
$P(H^0(F(c-a-b))^{\vee})$ is naturally identified with a 
linear subspace of the linear series $|\op3(c-b)|$.
\end{theorem}

\begin{proof}
(i) Consider the generalized null correlation bundle $[E]\in N$ and the corresponding monad \eqref{ein monad} with the cohomology sheaf $E$. From the description
\eqref{lambda,mu}-\eqref{C=S cap S'} of this monad it follows that there is a smooth complete intersection curve $C_0$ defined in \eqref{C=S cap S'} having the properties \eqref{OS(C)}-\eqref{quotient of N}. Besides, there is a well-defined double scheme structure $\overline{C}_0$ on $C_0$ 
satisfying the exact triple \eqref{s, alpha}, and another nonreduced scheme structure $C'_0$ on $C_0$ together with a 
zero-dimensional subscheme $Z$ of $C_0$, and these schemes
fit in the commutative diagram \eqref{diag with C_0}.  
By (\ref{s, alpha}), the composition $\beta\circ\alpha\circ s$ is zero, where $\beta$ is defined in \eqref{diag with C_0}. 
Hence the triple (\ref{s, alpha}) and the upper horizontal triple of the diagram \eqref{diag with C_0} extend to a commutative diagram
$$ 
\xymatrix{
& 0 \ar[d] & 0 \ar[d] & & \\
 & \op3(a+b-c) \ar[d]\ar@{=}[r] & \op3(a+b-c) \ar[d]^{s}  &  & \\
0 \ar[r] & F \ar[d] \ar[r] & E \ar[d]^{\alpha} 
\ar[r]^-{\beta\circ\alpha} &
\cali_{Z,S}(e-b) \ar@{=}[d] \ar[r] & 0 \\
0\ar[r] & \cali_{C_0}(e-a) \ar[d] \ar[r] & \cali_{\overline{C_0}}(c-a-b+e) \ar[d] \ar[r]^-{\beta} & \cali_{Z,S}(e-b) \ar[r]  & 0 \\
& 0 & 0. &  &} 
$$
The leftmost vertical triple of this diagram twisted by $\op3(c-b-a)$ coincides with (\ref{s for F}):
\begin{equation}\label{s for F with C0}
0\to \calo_{\p3}\to F(c-a-b)\to\cali_{C_0}(c-2a-b+e)
\to0.
\end{equation}
Since $C_0$ is a complete intersection (\ref{C=S cap S'}), it follows that the sheaf $\cali_{C_0}(c-2a-b+e)$ has the following locally free $\op3$-resolution:
\begin{equation}\label{resolution IC}
0\to\op3(e-c-a)\to\op3(e-2a)\oplus\op3(e-a-b)\to
\cali_{C_0}(c-2a-b+e)\to0. 
\end{equation}
Passing to sections in the triple (\ref{resolution IC}) and
(\ref{s for F with C0}) we obtain (\ref{h0(F...)=1}) and (\ref{h1(F...)=0}).

(ii) Note that, since by \eqref{h0(F...)=1} $h^0(F(c-a-b))\le3$,
it clearly follows that the zero-scheme $C=(s)_0$ of any non-zero section $s\in H^0(F(c-a-b))$ has dimension 1. (Indeed,
a standard argument in case $\dim C=1$ shows that there exists a positive integer $d$ and nonzero section $s'\in H^0(F(c-a-b-d))$ with $\dim(s')_0=1$, so that $\mathbf{k}s'\otimes H^0(\op3(d))$ is a subspace of dimension $\ge4$ of $H^0(F(c-a-b))$ which is a contradiction.) We thus have to treat 3 cases corresponding to the different values
of $h^0(F(c-a-b))$.

(ii.1) $h^0(F(c-a-b))=1$. Since by \eqref{s for F with C0} $C_0=
(s)_0$ for some $0\ne s\in H^0(F(c-a-b))$, it follows that, in
\eqref{s for F}, $C=C_0$ which is a complete intersection of
desired form \eqref{C=S cap S'}.

(ii.2) $h^0(F(c-a-b))=2$. In this case $e=a=0,\ b>0$, and for any 
$0\ne t\in H^0(F(c-b))$ the triple \eqref{s for F} becomes:
\begin{equation}\label{Ct}
0\to\op3\overset{t}{\to}F(c-b)\to\cali_{C_t}(c-b)\to0, 
\ \ \ \ \ C_t=(t)_0.
\end{equation} 
It follows that $h^0(\cali_{C_t}(c-b))=1$, i. e.
here exists a unique surface $S_t\in|\op3(c-b)|$ containing 
${C_t}$. Now the cokernel $Q$ of the evaluation morphism
$0\to H^0(F(c-b))\otimes\op3\xrightarrow{\mathrm{ev}}F(c-b)$
is by construction a sheaf supported on the divisor $S_t$ for
any $0\ne t\in H^0(F(c-b))$. Thus from the above uniqueness
we have $S_t=S$, where $S\in|\op3(c-b)|$ is a surface containing
the curve $C_0$.

Besides, passing to cohomology in the triples
\eqref{s for F with C0} and \eqref{resolution IC} twisted by 
$\op3(b)$ we obtain for $e=a=0$ that $h^0(F(c))$. This together with the triple \eqref{Ct} twisted by $\op3(b)$ yields $h^0(\cali_{C_t}(c))=2$. The last equality together with the
exact triple 
\begin{equation}\label{cali Ct}
0\to\op3(b)\xrightarrow{\cdot S}\cali_{C_t}(c)\to\calo_{S}
(-C_t)(c)\to0, 
\end{equation}
where $S$ is any surface of the 
implies $h^0(\calo_{S}(-C_t)(c))=1$. Since the sheaf $\calo_{S}(-C_t)(c)$ has degree 0 with respect to
$\calo_{S}(1)$, it follows from the last equality that 
$\calo_{S}(C_t)=\calo_{S}(c)$. Since $S\in|\op3(c-b)|$, it follows that $C_t$ is a complete intersection of the desired 
form \eqref{C=S cap S'}.

(ii.3) $h^0(F(c-a-b))=3$. In this case $e=a=b=0$ and, arguing as above in case (ii.2), we obtain for any $0\ne t\in H^0(F(c))$ that $h^0(\cali_{C_t}(c))=2$. Besides, for any surface $S\in|\op3(c-b)|$ passing through $C_t$, there is an exact triple \eqref{cali Ct} with $b=0$. This together with the last equality implies that $h^0(\calo_{S}(-C_t)(c))=1$, and as above we obtain that $C_t$ is a complete intersection curve of the form \eqref{C=S cap S'}.

Last, note that $C=(s)_0$ is smooth for a general
$s\in H^0(F(c-a-b))$, since $C_0$ is smooth. 

(iii) In case $h^0(F(c-a-b))\le2$, the assertion directly 
follows from (ii.1-2). Consider the case  $h^0(F(c-a-b))=3$. Note that, in this case, $a=b=e=0$. The exact triples 
$0\to\op3\to F(c)\to\cali_{C_0}(c)\to0$ and $0\to\op3(-c)\to2\op3\to\cali_{C_0}(c)\to0$ by push-out yield a resolution for $F(c)$ of the form $0\to\op3(-c)\to3\op3\to F(c)\to0$. This resolution shows that, for any 2-dimensional subspace $V$ of $H^0(F(c))$ the cokernel of the evaluation
morphism $0\to V\otimes\op3\xrightarrow{\mathrm{ev}}F(c)$ is
isomorphic to the sheaf $\calo_{S_t}$ for some surface $S_t\in|\op3(c)|$. These surfaces $S_t$ constitute a 2-dimensional linear subseries parametrized by $P(H^0(F(c-a-b))^{\vee})$. 
\end{proof}

Let $p:\p3 \times \mathbf{X}\to\mathbf{X}$ be the projection, and set
\begin{equation}\label{defn W}
\boldsymbol{\mathcal{W}}:=\mathbb{P}((p_*\mathbf{F}(c-a-b))^{\vee})
\xrightarrow{\boldsymbol{\pi}}\mathbf{X}.
\end{equation}
Note that, by (\ref{h0(F...)=1}), (\ref{h1(F...)=0}) and the
base change, $p_*\mathbf{F}(c-a-b)$ is a locally free sheaf
of rank
\begin{equation}\label{rk=}
\rk (p_*\mathbf{F}(c-a-b))=h^0(F(c-a-b)),
\end{equation}
where $h^0(F(c-a-b))$ is given in (\ref{h0(F...)=1}). Hence
$\boldsymbol{\pi}:\boldsymbol{\mathcal{W}}\to\mathbf{X}$ is a locally trivial projective bundle, and there is a canonical epimorphism of vector bundles on $\boldsymbol{\mathcal{W}}$
\begin{equation}\label{can epi}
\epsilon:\ \boldsymbol{\pi}^*((p_*\mathbf{F}(c-a-b))^{\vee})
\twoheadrightarrow\calo_{\boldsymbol{\mathcal{W}}}(1).
\end{equation}

Consider the $G$-action on $\mathbf{X}$ making the projection $\Phi:\mathbf{X}\to X$ a principal $G$-bundle (see Theorem \ref{Thm 2}(ii)). It follows from the definition of $\boldsymbol{\mathcal{W}}$ and Remark \ref{reduction step}(iii) that this action lifts to a $G$-action on $\boldsymbol{\mathcal{W}}$ such that $\boldsymbol{\pi}$ is a $G$-invariant morphism. We thus obtain a cartesian diagram of principal $G$-bundles 
\begin{equation}\label{diag W}
\xymatrix{ 
\boldsymbol{\mathcal{W}} \ar[d]_{\boldsymbol{\pi}}\ar[r]^{\tilde{\Phi}} & \mathcal{W} \ar[d]^{\pi} \\
\mathbf{X}\ar[r]^-{\Phi} & X,}
\end{equation}
where $\mathcal{W}=\boldsymbol{\mathcal{W}}//G$ is a geometric factor,
$\tilde{\Phi}:\boldsymbol{\mathcal{W}}\to\mathcal{W}$ is a canonical projection, and
$\pi:\mathcal{W}\to X$ is the induced morphism.

Let $\boldsymbol{\mathcal{W}}\xleftarrow{\tilde{p}}\p3\times
\boldsymbol{\mathcal{W}}\xrightarrow{\tilde{\boldsymbol{\pi}}}
\p3\times\boldsymbol{\mathcal{W}}$
be the induced projections. The canonical epimorphism 
$\epsilon$ from (\ref{can epi}) induces a morphism
\begin{equation}\label{bf s}
\mathbf{\tilde{s}}:\ \op3\boxtimes\calo_{\boldsymbol{\mathcal{W}}}(-1)
\xrightarrow{\tilde{p}^*(\epsilon^{\vee})}\tilde{p}^*\boldsymbol{\pi}^*p_*
\mathbf{F}(c-a-b)=\tilde{\boldsymbol{\pi}}^*p^*p_*
\mathbf{F}(c-a-b)\xrightarrow{\tilde{\boldsymbol{\pi}}^*ev}
\mathbf{F}_{\boldsymbol{\mathcal{W}}}(c-a-b).
\end{equation}
(Note that, here, $\mathbf{F}_{\boldsymbol{\mathcal{W}}}= \tilde{{\boldsymbol{\pi}}}^*\mathbf{F}$, according to our agreement on notation.)

\begin{theorem}\label{Thm 4}
(i) The variety $\mathcal{W}$ is described as $\mathcal{W}=\{(x,C)\ |\ x=(S,[E])\in X,$ and $C=(s)_0$ for some $0\ne s\in H^0(F(c-a-b)),$ where $F$ is determined by the pair $x=(S,[E])$ via the reduction step $(\ref{F from E}) \}$. In addition, the morphism $\pi:\mathcal{W}\to X$
is given by $(x,C)\mapsto x$, and $\pi^{-1}(x)=P(H^0(F(c-a-b)))$.

(ii) The vertical maps $\boldsymbol{\pi}:\boldsymbol{\mathcal{W}}\to\mathbf{X}$ and $\pi:\mathcal{W}\to X$ in \eqref{diag W} are locally trivial $\mathbb{P}^{m}$-fibrations, where 
\begin{equation}\label{formula for m}
m=m(e,a,b,c):=h^0(F(c-a-b))-1,
\end{equation} 
and $h^0(F(c-a-b))$ is given by \eqref{h0(F...)=1}. Therefore, 
$\dim \mathcal{W}=\dim X+m(e,a,b,c)$.
In particular, if $(e,a)\ne(0,0)$, then there is an isomorphism
$\pi:\mathcal{W}\xrightarrow{\simeq}X$.

(iii) There is an exact $\calo_{\p3\times\boldsymbol{\mathcal{W}}}$-triple
$0\to\op3\boxtimes\calo_{\boldsymbol{\mathcal{W}}}(-1)
\xrightarrow{\mathbf{\tilde{s}}}\mathbf{F}_{\boldsymbol{\mathcal{W}}}(c-a-b)\to\cali_{\boldsymbol{\tilde{\mathcal{C}}},
\p3\times\boldsymbol{\mathcal{W}}}(c-2a-b+e)\to0$,
where $\mathbf{\tilde{s}}$ is defined in \eqref{bf s} and
$\boldsymbol{\tilde{\mathcal{C}}}=(\mathbf{s})_0$ is a codimension 2 subscheme of $\p3\times\boldsymbol{\mathcal{W}}$. 
\end{theorem}

\begin{proof}
Statement  (i) follows from the base change and the definition of $\boldsymbol{\mathcal{W}}$ and $\mathcal{W}$. In (ii), the local triviality of the fibration $\boldsymbol{\pi}$ is clear, and Theorem \ref{Thm 3}(iii) yields the local triviality of the fibration $\pi$.
The isomorphism (\ref{isom pi}) is a corollary of (\ref{h0(F...)=1}). Statement (iii) follows from the definition of the morphism $\mathbf{s}$ in (\ref{bf s}).
\end{proof}

Now consider the dense open subset $W$ of $\mathcal{W}$ defined in \eqref{W as a set}:
\begin{equation}\label{W in calW}
\xymatrix{
W=\{([E],S,C)\in\mathcal{W}\ |\ C\ is\ smooth\}\ \ar@{^{(}->}[rrr]^-{\mathrm{dense\ open}} & & & \mathcal{W}},
\end{equation}
and set
\begin{equation*}
\mathbf{W}:=\boldsymbol{\mathcal{W}}\times_{\mathcal{W}}W
\xrightarrow{\boldsymbol{\pi}}\mathbf{X}, \ \ \ \ \ 
\mathbf{F}_{\mathbf{W}}:=
\mathbf{F}_{\boldsymbol{\mathcal{W}}}|_{\p3\times\mathbf{W}},\ \ \ \boldsymbol{\mathcal{C}}:=\boldsymbol{\tilde{\mathcal{C}}}
\times_{\mathcal{W}}W.
\end{equation*}
In view of Theorem \ref{Thm 3}(ii) the morphisms $\mathbf{W}\xrightarrow{\pi}\mathbf{X}$ and $W\xrightarrow{\pi}X$ are surjective. Thus from Theorem \ref{Thm 4} we obtain
\begin{corollary}\label{cor 5.3}
(i) $\mathbf{W}\xrightarrow{\pi}\mathbf{X}$ and $W\xrightarrow{\pi}X$ are open subfibrations of locally trivial
$\mathbb{P}^{m}$-fibrations, where $m=m(e,a,b,c)$ is defined in \eqref{formula for m}, and
\begin{equation}\label{dim W tau}
\dim \mathcal{W}=\dim X+m(e,a,b,c).
\end{equation}
In particular, if $(e,a)\ne(0,0)$, then there is an isomorphism
\begin{equation}\label{isom pi}
\pi:W\xrightarrow{\simeq}X.
\end{equation}

(ii) There is an exact $\calo_{\p3\times\mathbf{W}}$-triple
\begin{equation}\label{global C}
0\to\op3\boxtimes\calo_{\mathbf{W}}(-1)
\xrightarrow{\mathbf{s}}\mathbf{F}_{\mathbf{W}}(c-a-b)\to\cali_{\boldsymbol{\mathcal{C}},
\p3\times\mathbf{W}}(c-2a-b+e)\to0,
\end{equation}
where $\mathbf{s}=\mathbf{\tilde{s}}|_{\p3\times\mathbf{W}}$.
This triple being restricted onto $\p3\times\{\mathbf{w}\}$, for an arbitrary point $\mathbf{w}\in\mathbf{W}$, coincides with the triple (\ref{s for F}), in which we set $s=\mathbf{s}\otimes\mathbf{k}(\mathbf{w})$,
$F=\mathbf{F}_{\mathbf{W}}|_{\p3\times\{\mathbf{w}\}}$,
and $C=\boldsymbol{\mathcal{C}}\cap\p3\times
\{\mathbf{w}\}$.
\end{corollary}

\begin{remark}\label{W reduced}
\textit{According to Theorem \ref{Thm 2}(ii) and the above Corollary, $\mathbf{W}\xrightarrow{\boldsymbol{\pi}}
\mathbf{X}\xrightarrow{\boldsymbol{\theta}}Y
\xrightarrow{\varphi}N$ is a composition of two open subfibrations of
projective fibrations and of a principal bundle. Hence, since
$N$ is a reduced scheme by \cite{Ein}, it
follows that $\mathbf{W}$ is a reduced scheme.\\
(ii) Applying the functor $\tilde{{\boldsymbol{\pi}}}^*$ to the epimorphism $\boldsymbol{\psi}$ in (\ref{E from F}) we obtain an epimorphism 
${\boldsymbol{\psi}}_{\mathbf{W}}:\ \mathbf{F}_{\mathbf{W}}
\twoheadrightarrow(\mathbf{L}_{\mathbf{\Gamma}}^{\vee}(b,1))_{\mathbf{W}}$, hence also an epimorphism
\begin{equation}\label{psi W}
{\boldsymbol{\psi}}_{\mathbf{W}}:\ \mathbf{F}_{\mathbf{W}}
\twoheadrightarrow(\mathbf{L}_{\mathbf{\Gamma}}^{\vee}(b,1))_{\mathbf{W}}.
\end{equation}
}
\end{remark}

\vspace{5mm}

\section{A new family $\underline{\mathbf{F}}$ of reflexive sheaves} \label{section 5}

\vspace{2mm}

In this section we construct a new family $\underline{\mathbf{F}}$ of reflexive sheaves with Chern classes (\ref{ci(F) 0}) and with the same properties as that of the sheaves of the family $\mathbf{F}$
-- see Theorem \ref{Thm 7}. As above, we fix the numbers $e,a,b,c$ which determine an Ein component $\overline{N}$ of $M(e,n)$, where $n=c^2-a^2-b^2-e(c-a-b)$. 
Take an arbitrary point $(S,C)\in\mathrm{R}$ and compute the
number $h^0(\calo_C(c+a-e))$. Since $C$ is a complete intersection curve $C=S\cap S'$ (see Remark \ref{complete int}(i)), we obtain the equality
\begin{equation}\label{det NC}
\det N_{C/\p3}=\calo_C(2c-a-b)
\end{equation}
and the exact triples
$$
0\to\cali_C(c+a-e)\to\calo_{\p3}(c+a-e)\to\calo_C(c+a-e)\to0,
$$
$$
0\to\calo_{\p3}(2a+b-c-e)\to\calo_{\p3}(a+b-e)\oplus
\calo_{\p3}(2a-e)\to\cali_C(c+a-e)\to0.
$$
These triples yield
\begin{equation}\label{h0=}
h^0(\calo_C(c+a-e))=\binom{c+a-e+3}{3}-\binom{a+b-e+3}{3}-
\binom{2a-e+3}{3}+\delta(e,a,b,c),
\end{equation}
where
\begin{equation}\label{delta}
\delta(e,a,b,c)=\left\{
\begin{array}{cc}
\binom{2a+b-c-e+3}{3}, & \mathrm{if}\ c\le 2a+b-e, \\
& \\
0, &  \mathrm{if}\ c>2a+b-e.
\end{array}
\right.
\end{equation}

For an arbitrary point $x=(S,C)\in\mathrm{R}$ consider the groups
$$
\mathrm{Ext}^i(x):=\mathrm{Ext}^i(\cali_C(c-2a-b+e),\calo_{\p3}), \ \ \ i=0,1.
$$
From (\ref{det NC}) it follows that
\begin{equation}\label{local ext}
\begin{split}
&\cale xt^1(\cali_C(c-2a-b+e),\calo_{\p3})=
\cale xt^2(\calo_C(c-2a-b+e),\calo_{\p3})=\\
& \det N_{C/\p3}(2a+b-c-e)\simeq\calo_C(c+a-e).
\end{split}
\end{equation}
Since $h^i(\calh om(\cali_C(c-2a-b+e),\calo_{\p3}))=
h^i(\op3(2a+b-c-e))=0,\ i=0,1,2,$ from (\ref{h0=}), (\ref{local ext}) and the spectral sequence of local-to-global Ext's we obtain
\begin{equation}\label{dim Ext0}
\dim\mathrm{Ext}^0(x)=h^0(\op3(2a+b-c-e)),
\end{equation}

\begin{equation}\label{dim Ext1}
\begin{split}
& \dim\mathrm{Ext}^1(x)=h^0(\calo_C(c+a-e))=\\
& \binom{c+a-e+3}{3}-\binom{a+b-e+3}{3}-
\binom{2a-e+3}{3}+\delta(e,a,b,c).
\end{split}
\end{equation}

\begin{remark}\label{sheaves cale i}
Consider the incidence subvariety $\Sigma$ of $\p3\times\mathrm{R}$ defined in \eqref{def of Sigma}.
In view of (\ref{dim Ext0})--(\ref{dim Ext1}) the dimensions of the groups $\mathrm{Ext}^1(x)$ do not depend on the point 
$x=(S,C)\in\mathrm{R},$ so that the sheaves
$$
\cale_i:=\cale xt^i_{p_2}(\cali_{\Sigma,\p3\times\mathrm{R}}
(c-2a-b+e),\calo_{\p3\times\mathrm{R}}),\ \ i=0,1,
$$
by \cite{BPS} commute with the base change in the sense of \cite[Remark 1.5]{L}. In particular,
the sheaf $\cale_1$ is a locally free $\calo_{\mathrm{R}}$-sheaf of rank 
\begin{equation*}
\rk\cale_1=h^0(\calo_C(c+a-e))
\end{equation*}
and for any $x=(S,C)\in\mathrm{R}$ one has the base change
isomorphism $\cale_1\otimes\mathbf{k}(x)\overset{\simeq}{\to}
\mathrm{Ext}^1(x)$.
\end{remark}

Consider the rational variety 
\begin{equation}\label{T}
\mathbf{T}:=\mathbb{P}(\cale_1^{\vee})
\end{equation} 
with its structure morphism $\boldsymbol{\mu}:\mathbf{T}\to\mathrm{R}$ which is a locally trivial projective fibration with fibre of dimension $h^0(\calo_C(c+a-e))-1$. We thus obtain from (\ref{dim R}) and (\ref{dim Ext1}) the formula for the dimension of $\mathbf{T}$:
\begin{equation}\label{dim T}
\begin{split}
& \dim \mathbf{T}=\binom{c-b+3}{3}+\binom{c-a+3}{3}-\binom{b-a+3}{3}+
\binom{c+a-e+3}{3}\\
&-\binom{a+b-e+3}{3}-\binom{2a-e+3}{3}+\delta(e,a,b,c)-3.
\end{split}
\end{equation}
By construction, $\mathbf{T}$ has a set-theoretical description \eqref{bfT as a set}, and the structure morphism
$\boldsymbol{\mu}$ coincides with \eqref{set-th def of mu}.
In addition, each point $t=(S,C,\xi)\in\mathbf{T}$ defines a non-trivial (class of proportionality of an) extension of $\calo_{\p3}$-sheaves
\begin{equation}\label{extension t}
\xi:\ \ \ \ \ 0\to\calo_{\p3}(a+b-c)\to F_t\to\cali_C(e-a)\to0. 
\end{equation}
\begin{remark}\label{Serre constr}
This is the well-known Serre construction -- cf. \cite{H-vb}, \cite{H-r}, \cite{OSS}. In particular, $F_t$ is a reflexive sheaf with Chern classes given by \eqref{ci(F) 0}.
\end{remark}

Globalizing over $\mathbf{T}$ the triple \eqref{extension t} we obtain the following result.

\begin{theorem}\label{Thm 7}
On $\p3\times \mathbf{T}$ there is a sheaf $\underline{\mathbf{F}}$ defined as the universal extension sheaf \eqref{univ FT}.
The sheaf $\underline{\mathbf{F}}$ is a family of reflexive sheaves \eqref{extension t} on $\mathbb{P}^3$ with the base $\mathbf{T}$. 
\end{theorem}
In the remaining part of this section we study the question of producing a generalized null correlation bundle $E$ from an arbitrary reflexive sheaf $F$ of the family $\underline{\mathbf{F}}$.
A hint for this is given by the triple (\ref{E from F restricted}). In this triple a generalized null correlation bundle $E$ is obtained from $F$ by an analogue of the "inverse reduction step" (cf. Remark \ref{reduction step}(i)) as a kernel of an epimorphism $F\twoheadrightarrow\mathcal{O}_S(b)$. In fact, the following theorem is true which will be used in
the next Section.

\begin{theorem}\label{Thm 10}
Consider a subset $T$ of $\mathbf{T}$ consisting of those points $t=(S,C,\xi)\in\mathbf{T}$ for which there exists an epimorphism $\psi:\ F_t\twoheadrightarrow\mathcal{O}_S(b)$, with $F_t$ given by an extension (\ref{extension t}), such that $E=\ker\psi$ is locally free. Then $T$ is nonempty and $E$ is a generalized  null correlation bundle, $[E]\in N_{\mathrm{nc}}$.
\end{theorem}
\begin{proof} Clearly, $T$ is nonempty: it is enough to take a point $\mathbf{x}=(S,y)\in\mathbf{X}$ and set $[E]=\varphi(y)$,
so that the data $(F,C,\xi)$ are determined by the pair $\bigl(S,[E]\bigr)$
as in Theorem \ref{Thm 3}; in particular, $\xi$ is defined as the extension class of the triple (\ref{s for F}). Then for the point $t=(S,C,\xi)$ by (\ref{s for F}) the sheaf $F_t=F$ coincides with the sheaf $F_{\mathbf{x}}$ in the triple (\ref{E from F restricted}), and this triple shows that $t\in T$. 

Now take $t=(S,C,\xi)\in T$ and consider the triple (\ref{E from F restricted}) twisted by  $\calo_{\p3}(c-b+m)$:
\begin{equation}\label{again exact E,F}
0\to E(m)\to F(c-b+m)\overset{\psi}{\to}\calo_S(c+m)\to0,\ \ \  m\in\mathbb{Z}.
\end{equation} 
Respectively, the triple (\ref{extension t}) twisted by $\calo_{\p3}(c-b+m)$ yields
\begin{equation}\label{new extension t}
0\to\calo_{\p3}(a+m)\overset{i}{\to}F(c-b+m)\overset{\theta}{\to}\cali_C(m+c-a-b+e)\to0.\end{equation}
Besides we have a standard exact triple
\begin{equation}\label{res of I}
\begin{split}
& 0\to\calo_{\p3}(-c+e+m)\to\calo_{\p3}(-b+e+m)
\oplus\calo_{\p3}(-a+e+m)\to \\
& \cali_C(m+c-a-b+e)\to0.
\end{split}
\end{equation}
Substituting $m\le b-c$ into (\ref{new extension t}) and 
(\ref{res of I}) and using the inequalities $c>a+b,\ e\le0$ we obtain $h^0(F(m))\le0,\ m\le0$. Besides, since $Z\ne\emptyset$ and
$e-b\le0$, $h^0(\cali_{Z,S}(e-b+m))=0,\ m\le0$. Hence the triple
(\ref{F from E}) twisted by  $\calo_{\p3}(m)$ implies
\begin{equation}\label{E stable}
h^0(E(m))=0,\ \ \ m\le0.
\end{equation}
In particular, $h^0(E)=0$, i.~e. $E$ is stable. 

Now consider the triples (\ref{again exact E,F}) and (\ref{new extension t})
and the morphisms $\psi$ and $i$ therein. If the composition $\psi\circ i$ is zero then $i$ becomes a section of $E$ which contradicts (\ref{E stable}). Hence, the composition
$\psi\circ i$ factors as
$$
\psi\circ i:\  \calo_{\p3}(a+m)\overset{\psi'}{\twoheadrightarrow}
\calo_{S}(a+m)\overset{i'}{\to}\calo_S(c+m).
$$
Denote $U=\p3\setminus\mathrm{Sing}F$.
Since by (\ref{new extension t}) the morphism $i|_U:\calo_U(a+m)\to
F(c-b+m)|_U$ is a section of the locally free sheaf $F(c-b+m)|_U$
vanishing at the curve $C\cap U$, it follows that $i':\calo_{S}(a+m)\to\calo_S(c+m)$
is a multiplication by the equation of the divisor $C$ in $S$. 
Hence $\coker i'=\calo_C(c+m)$ and we obtain a commutative diagram
$$ 
\xymatrix{
& 0 \ar[d] & 0 \ar[d] & 0 \ar[d] & \\
0 \ar[r] & \calo_{\p3}(a+b-c+m) \ar[d]\ar[r]^-{\cdot S} & \calo_{\p3}(a+m) \ar[d]^-{i} \ar[r]^{\psi'} & \calo_S(a+m) \ar[d]^-{i'} \ar[r] & 0 \\
0 \ar[r] & E(m) \ar[d] \ar[r] &
F(c-b+m) \ar[d]^{\theta} \ar[r]^-{\psi} &
\calo_S(c+m) \ar[d] \ar[r] & 0 \\
0\ar[r] & \cali_{\tilde{C}}(c-a-b+e+m) \ar[d] \ar[r] & \cali_C(c-a-b+e+m) \ar[d] \ar[r]^-{\psi''} & \calo_C(c+m) \ar[r] \ar[d] & 0 \\
& 0 & 0 & 0, &} 
$$
in which $\psi''$ is induced by the morphisms $\psi$ and $\psi'$, and $\tilde{C}$ is a certain double scheme structure on the curve $C$.
Consider the bottom horizontal and left vertical triples in this diagram:
\begin{equation}\label{tilde C}
0\to\cali_{\tilde{C}}(c-a-b+e+m)\to\cali_C(c-a-b+e+m)\overset{\psi''}{\to}\calo_C(c+m)\to0,
\end{equation}
\begin{equation}\label{E, tilde C}
0\to\calo_{\p3}(a+b-c+m)\to E(m)\to\cali_{\tilde{C}}(m+c-a-b+e)\to0.
\end{equation}

The triple (\ref{tilde C}) by \cite[Example 3.3]{Rao} shows that 
the cohomology $H^0_*(\calo_{\p3})$-module $H^1_*(\cali_{\tilde{C}})$
as a graded module over the graded ring 
$H^0_*(\calo_{\p3})\simeq\mathbf{k}[x_0,x_1,x_2,x_3]$
has one generator. Hence the triple (\ref{E, tilde C}) implies that the cohomology $H^0_*(\calo_{\p3})$-module $H^1_*(E)$ also has one generator. This together with \cite[Prop. 1.3]{Ein} shows that $E$ is a generalized null correlation bundle.
\end{proof}

\vspace{2mm}

\section{Family of generalized null correlation bundles $\underline{\mathbf{E}}$ associated to $\underline{\mathbf{F}}$} \label{section 6}

\vspace{2mm}
In this section, starting with the family $\underline{\mathbf{F}}$ with rational base  $\mathbf{T}$, we 
produce a family $\underline{\mathbf{E}}$ of generalized null correlation bundles with certain rational base $V$. For this, 
we first prove Theorem \ref{Thm 11} in which we state certain
properties of the restriction of a reflexive sheaf $F_t$ of the
family $\underline{\mathbf{F}}$ onto a surface $S$, where $t=
(S,C,\xi)\in\mathbf{T}$. From these properties it follows that
the set $T$ from Theorem \ref{Thm 10} is a dense open subset 
$T$ of $\mathbf{T}$. Theorem \ref{Thm 11} then also leads to a 
construction of a desired rational family $V$ as of a dense 
open subset of a locally trivial projective fibration over $T$.

\begin{theorem}\label{Thm 11}
In conditions and notation of Theorem \ref{Thm 10}, let $t=(S,C,\xi)\in T$ and $F=F_t$. Then the following statements hold.

(i) $\dim\mathrm{Hom}(F,\calo_S(b))=\binom{b+c-e+3}{3}-
\binom{2b-e+3}{3}+1$.

(ii) the set $P(\mathrm{Hom}(F,\calo_S(b)))^*=\{\mathbf{k}\psi\in
P(\mathrm{Hom}(F,\calo_S(b)))\ |\ \psi:\ F\to\calo_S(b)$ \\ $\mathrm{is\ surjective\ and}\ \ker\psi\ \mathrm{is\ locally\ free}\}$ is nonempty, hence dense open in\\ $P(\mathrm{Hom}(F,\calo_S(b)))$.

(iii) For any point $\mathbf{k}\psi\in
P(\mathrm{Hom}(F,\calo_S(b)))^*$, the sheaf
$$
E_{\psi}:=(\ker(F\twoheadrightarrow\calo_S(b)))(c-b)
$$
is a generalized null correlation bundle, $[E_{\psi}]\in N_{\mathrm{nc}}$.
\end{theorem}

\begin{proof}
(i) Note that the natural epimorphism 
$\rho:\cali_C(e-a)\twoheadrightarrow\cali_{C,S}(e-a)\simeq
\calo_S(-C)(e-a)\simeq\calo_S(e-c)$ composed with the
epimorphism $\theta:F\ \twoheadrightarrow\cali_C(e-a)$ from the
triple (\ref{new extension t}) for $m=b-c$ gives an epimorphism
$$
\rho\circ\theta:F\ \twoheadrightarrow\calo_S(e-c).  
$$ 
Restricting it onto $S$ yields an exact triple:
$0\to\cali_{Z,S}(b)\to F|_S\to\calo_S(e-c)\to0$.
This triple together with the triple 
$0\to\cali_{Z,S}(b)\to\calo_S(b)\to\calo_Z\to0$
by push-out yield two exact triples:
\begin{equation}\label{1st triple FS}
0\to\calo_S(b)\overset{u}{\to}(F|_S)^{\vee\vee}\to\calo_S(e-c)\to0,
\end{equation}
\begin{equation}\label{F to Fdd}
0\to F|_S\to(F|_S)^{\vee\vee}\to\calo_Z\to0,
\end{equation}
where $(F|_S)^{\vee\vee}=\calh om_{\calo_S}(\calh om_{\calo_S}(F|_S,\calo_S),\calo_S)$. On the other hand, restricting onto $S$ the epimorphism 
$\psi:\ F\twoheadrightarrow\calo_S(b)$ from the triple
(\ref{again exact E,F}) with $m=b-c$ we obtain an exact triple
$0\to\cali_{Z,S}(e-c)\to F|_S\to\calo_S(b)\to0$. As above, by push-out this triple yields an exact triple
\begin{equation}\label{2nd triple FS}
0\to\calo_S(e-c)\to (F|_S)^{\vee\vee}\overset{v}{\to}
\calo_S(b)\to0.
\end{equation}
Now consider the morphisms $u$ and $v$ in the triples
(\ref{1st triple FS}) and (\ref{2nd triple FS}). If their
composition $v\circ u:\calo_S(b)\to\calo_S(b)$ iz zero, this implies that there exists a nonzero morphism $\calo_S(b)\to\calo_S(e-c)$, contrary to the condition that
$e-c-b<0$. Hence  $v\circ u:\calo_S(b)\to\calo_S(b)$ is an isomorphism. This means that both triples (\ref{1st triple FS})   and (\ref{2nd triple FS}) split. Thus
\begin{equation}\label{split}
(F|_S)^{\vee\vee}\simeq\calo_S(b)\oplus\calo_S(e-c).
\end{equation}
Remark that, since $\dim Z=0$, it follows that $\mathrm{Hom}(\calo_Z,\calo_S(b))=\mathrm{Ext}^1(\calo_Z,
\calo_S(b))=0$, the triple (\ref{F to Fdd}) yields the isomorphisms 
$\mathrm{Hom}(F,\calo_S(b))\simeq\mathrm{Hom}(F|_S,
\calo_S(b))\simeq\mathrm{Hom}((F|_S)^{\vee\vee},
\calo_S(b))$. This together with (\ref{split}) shows that
$$
\mathrm{Hom}(F,\calo_S(b))=H^0(\calo_S)\oplus H^0(\calo_S(b+c-e)).
$$
Whence, (i) follows.

Statements (ii) and (iii) are immediate consequences of Theorem \ref{Thm 10}.
\end{proof}

Now return to the family $\underline{\mathbf{F}}$ of reflexive sheaves on $\p3\times\mathbf{T}$, and recall that $\mathbf{T}$ is a rational variety (see (\ref{T})) 
with the projection $r\circ\mu:\mathbf{T}\to\mathbf{P}$. Let 
$\underline{\Gamma}:=\Gamma\times_{\mathbf{P}}\mathbf{T}\subset\p3\times\mathbf{T}$ be the family of surfaces in $\p3$ with base
$\mathbf{T}$, together with the natural projection 
$\underline{\Gamma}\to\mathbf{T}$, the fibre of which over an
arbitrary point $t=(S,C,\xi)\in\mathbf{T}$ is a surface $S$.
Consider an $\calo_{\mathbf{T}}$-sheaf 
$$
\mathcal{A}=\cale xt^0_{\mathrm{pr_2}}(\underline{\mathbf{F}},
\calo_{\underline{\Gamma}}(b)),
$$
where $\mathrm{pr_2}:\p3\times\mathbf{T}\to\mathbf{T}$ is the projection. The base change and Theorem \ref{Thm 11}(i) show that
\begin{equation}\label{A times k(t)}
\mathcal{A}\otimes\mathbf{k}(t)=\mathrm{Hom}
(F_t,\calo_S(b)),\ \ \ F_t=\underline{\mathbf{F}}|_{\p3\times\{t\}},
\end{equation}
and $\mathcal{A}$ is a locally free $\calo_{\mathbf{T}}$-sheaf of
rank
\begin{equation}\label{rkA}
\rk\mathcal{A}=\binom{b+c-e+3}{3}-\binom{2b-e+3}{3}+1.
\end{equation}
Since $\mathbf{T}$ is a rational variety, the scheme 
\begin{equation}\label{bf B}
\mathbf{U}:=\mathbb{P}(\mathcal{A}^{\vee})
\xrightarrow{\boldsymbol{\lambda}}\mathbf{T}
\end{equation}
is a rational variety and its structure morphism
${\boldsymbol{\lambda}}:\mathbf{U}\to\mathbf{T}$
is a locally trivial projective fibration with fibre of dimension $\rk{\mathcal{A}}-1$. Thus by (\ref{dim T}) and (\ref{rkA}):
\begin{equation}\label{dim B}
\begin{split}
& \dim\mathbf{U}=\binom{c-b+3}{3}+\binom{c-a+3}{3}-
\binom{b-a+3}{3}+\binom{c+a-e+3}{3}\\
&-\binom{a+b-e+3}{3}-\binom{2a-e+3}{3}+\binom{b+c-e+3}{3}-
\binom{2b-e+3}{3}+\delta(e,a,b,c)-3.
\end{split}
\end{equation}
In view of (\ref{A times k(t)}) we have the set-theoretic description of $\mathbf{U}$ as:
\begin{equation}\label{descr of bf B}
\mathbf{U}=\{(t,\mathbf{k}{\psi})\ |\ t=(S,C,\xi)\in\mathbf{T},\ \mathbf{k}\psi\in P(\mathrm{Hom}(F_t,\calo_S(b)))\}.
\end{equation}
On $\mathbf{U}$ there is a tautological subbundle morphism 
\begin{equation*}
j:\ \calo_{\mathbf{U}}\to\mathbf{A}
\otimes\calo_{\mathbf{U}}(1), 
\end{equation*}
where $\calo_{\mathbf{U}}(1)$ is the Grothendieck sheaf and $\mathbf{A}:=\boldsymbol{\lambda}^*\mathcal{A}$. From Theorem \ref{Thm 11}(ii) it follows that 
\begin{equation}\label{B*}
U=\{(t,\mathbf{k}\psi)\in\mathbf{U}\ |\ \mathbf{k}\psi\in P(\mathrm{Hom}(F_t,\calo_S(b)))^*\}
\end{equation}
is a nonempty open (hence dense) subset of $\mathbf{U}$.
Since $\boldsymbol{\lambda}:\mathbf{U}\to\mathbf{T}$ is a
projective fibration, it is flat. Hence by the openness of flat morphisms (see, e. g., \cite[Ch. III, Exc. 9.1]{H})  
the set $T=\lambda(U)$ is a nonempty open (hence dense) subset of $\mathbf{T}$. We now set
\begin{equation}\label{notation for B}
\begin{split}
&  
\Gamma_{\mathbf{U}}:=\underline{\Gamma}\times_{\mathbf{T}}\mathbf{U},\ \ \ \ 
\Gamma_U:=\Gamma_{\mathbf{U}}\times_{\mathbf{U}}U=
\underline{\Gamma}\times_TU,\\ 
& \ \ \ \ \ A:=\mathbf{A}_U, \ \ \ \ \ \calo_B(1):=(\calo_{\mathbf{U}}(1))_U,
\ \ \ \ \lambda:=\boldsymbol{\lambda}|_U,
\end{split}
\end{equation}
and let 
\begin{equation*}
\mathrm{can}:\underline{\mathbf{F}}_{\mathbf{U}}\otimes\op3\boxtimes\mathbf{A}\to\calo_{\Gamma_{\mathbf{U}}}(b)
\end{equation*}
be the canonical evaluation morphism. Consider the universal
morphism
\begin{equation}\label{tilde Psi}
\mathbf{\Psi}:\ \underline{\mathbf{F}}_{\mathbf{U}}\to
\calo_{\Gamma_{\mathbf{U}}}(b)\otimes\calo_{\p3}\boxtimes\calo_{\mathbf{U}}(1)
\end{equation}
defined as the composition
\begin{equation*}
\mathbf{\Psi}:\ \underline{\mathbf{F}}_{\mathbf{U}}
\xrightarrow{\mathrm{id}\otimes j}\underline{\mathbf{F}}_{\mathbf{U}}
\otimes\op3\boxtimes(A\otimes\calo_{\mathbf{U}}(1))
\xrightarrow{\mathrm{can}\otimes\mathrm{id}}
\calo_{\Gamma_{\mathbf{U}}}(b)\otimes
\calo_{\p3}\boxtimes\calo_{\mathbf{U}}(1).
\end{equation*}
By Theorem \ref{Thm 11}(iii),
\begin{equation}\label{B=}
\begin{split}
& U=\{\mathbf{u}=(t,\mathbf{k}\psi)\in\mathbf{U}\ |\ t=(S,C,\xi)\in T,\ \ \mathbf{\Psi}|_{\p3\times\{\mathbf{u}\}}:
\ \underline{\mathbf{F}}\otimes\mathbf{k}(t)\to\calo_S(b)\\ 
& is\ surjective\ and\ [\ker\boldsymbol{\Psi}(c-b)|_{\p3\times\{\mathbf{u}\}}]\in N_{\mathrm{nc}}\}
\end{split}
\end{equation}
is a dense open subset of $\mathbf{U}$, and we obtain a 
well-defined morphism
\begin{equation}\label{morphism Phi}
q:\ U\to N_{\mathrm{nc}},\ \ \ \mathbf{u}\mapsto
[\ker\boldsymbol{\Psi}(c-b)|_{\p3\times\{\mathbf{u}\}}].
\end{equation}
Set
\begin{equation}\label{underline E}
V:=q^{-1}(N),\ \ \ \ \ \ 
\underline{\mathbf{E}}:=(\ker \boldsymbol{\Psi})(c-b)|\p3
\times V.
\end{equation}

\begin{remark}\label{global triple E,F}
(i) Note that $V$ is nonempty. Indeed, for any $([E],S,C)\in W$ the point $f([E],S,C)$ defined in \eqref{set-theoretic def of f} belongs to $V$.\\
(ii) Since $N$ is a dense open subset of 
$N_{\mathrm{nc}}$, it follows that $V$ is a dense 
open subset of $U$, hence also of $\mathbf{U}$, i. e. $V$ 
is a rational variety of dimension given by formula 
\eqref{dim B}. In addition, comparing \eqref{dim B} with 
\eqref{dim N 1} we obtain
\begin{equation}\label{dim B=}
\dim V=\dim N+\delta(e,a,b,c)+t(e,a,b), 
\end{equation}
and from \eqref{dim X} and \eqref{dim B=} it follows that
\begin{equation}\label{dim X-dim B}
\dim X-\dim V=\tau-\delta(e,a,b,c)-t(e,a,b).
\end{equation}
(iii) Clearly, from \eqref{underline E} follows the exact triple \eqref{tilde E from tilde F}. 
\end{remark}

\vspace{5mm}

\section{Relation between $\underline{\mathbf{E}}$ and $\mathbf{E}$. Proof of the main theorem} \label{section 7}
 
\vspace{2mm}

We are now ready to prove the main result of the paper, Theorem
\ref{main Thm}, which follows from the relation between the families $\underline{\mathbf{E}}$ and $\mathbf{E}$. (The exact form of this relation is the isomorphism (\ref{E from tilde E}).) 

\begin{theorem}\label{main Thm}
(i) There is an isomorphism of varieties
\begin{equation}\label{main isom}
f: W\overset{\simeq}{\to}V.
\end{equation} 
(ii) For $e,a,b,c$ with $e\in\{-1,0\}$ and $b\ge a\ge0,\ c>a+b$, the variety $N$, hence also the variety $\overline{N}=\overline{N}(e,a,b,c)$ is at  least stably 
rational. Furthermore, on $\p3\times W$ there 
exists a family of generalized null correlation bundles 
$\underline{\mathbf{E}}_{W}$ for which the corresponding 
modular morphism $W\to N$ coincides with 
$\theta\circ\pi$ in diagram \eqref{diag for X,W}.

(iii) Assume $(e,a)\ne(0,0),\ c>2a+b-e$, and $b>a$. Then $\tau=0$, $\overline{N}$ is a rational variety, and its open dense subset $N\simeq X\simeq W$ is a fine moduli space, i.e. the $\calo_{\p3\times N}$-sheaf $\underline{\mathbf{E}}_{W}$ is a universal family of generalized null correlation bundles over $N$.
\end{theorem}

\begin{proof}
(i) The desired map $f:W\to V$ was set-theoretically defined 
in \eqref{set-theoretic def of f}. We have to show that this
is the underlying map of a certain morphism. We first 
construct a $PGL(N_m)$-invariant morphism
\begin{equation}\label{bf f}
\mathbf{f}_V:\ \mathbf{W}\to V.
\end{equation}

For this, consider the triple (\ref{global C}) and remark that
the subscheme $\boldsymbol{\mathcal{C}}$ in this triple is a family with base $\mathbf{W}$ of complete intersection curves from $\mathrm{R}$ (see (\ref{bfR=})). Thus by the universality of the Hilbert scheme there exists a morphism 
$\mathbf{f}_0:\ \mathbf{W}\to\mathrm{R}$ such that $\boldsymbol{\mathcal{C}}=\Sigma\times_{\mathrm{R}}\mathbf{W}$. 
Hence, 
$$
\cali_{\boldsymbol{\mathcal{C}},\p3\times\mathbf{W}}
(c-2a-b+e)\simeq(\mathrm{id}_{\p3}\times\mathbf{f}_0)^*
\cali_{\Sigma,\p3\times\mathrm{R}}(c-2a-b+e).
$$ 
Now consider the triples (\ref{global C}) and 
(\ref{univ FT}) as families of extensions of $\op3$-sheaves with bases $\mathbf{W}$ and $\mathbf{T}$, respectively. Use Remark \ref{sheaves cale i} and the fact that $\mathbf{W}$ is reduced (see Remark \ref{W reduced}) to apply the universal property of the scheme $\mathbf{T}$ (see \cite[Cor. 4.4]{L}). By this universal property there is a uniquely defined morphism 
$\mathbf{f}_1:\ \mathbf{W}\to\mathbf{T}$ such that 
$\mathbf{f}_0=\boldsymbol{\mu}\circ\mathbf{f}_1$ and such that the triple (\ref{global C}) is obtained by applying the functor 
$(\mathrm{id}_{\p3}\times\mathbf{f}_1)^*$ to the triple
(\ref{univ FT}). In particular,
\begin{equation}\label{F from FX}
\mathbf{F}_{\mathbf{W}}\simeq
\underline{\mathbf{F}}_{\mathbf{W}}.
\end{equation}
By (\ref{F from FX}) and the universal property of the scheme $\mathbf{U}$ over $\mathbf{T}$ there is a unique morphism $\mathbf{f}_V:\ \mathbf{W}\to\mathbf{V}$ such that
$\mathbf{f}_1=\boldsymbol{\lambda}\circ\mathbf{f}_V$ and such that the epimorhism 
$\boldsymbol{\psi}_{\mathbf{W}}:\ 
\mathbf{F}_{\mathbf{W}}
\twoheadrightarrow(\mathbf{L}_{\rho}^{\vee}(b,1))
_{\mathbf{W}}$
in (\ref{psi W}) is obtained from the universal morphism
$\mathbf{\Psi}$ in (\ref{tilde Psi}) by aplying the functor
$(\mathrm{id}_{\p3}\times\mathbf{f}_V)^*$. As $\boldsymbol{\psi}_{\mathbf{W}}$
is surjective, from the description \eqref{underline E} of $V$ it follows that
\begin{equation*}
\mathbf{f}_V(\mathbf{W})\subset V
\end{equation*}
and $\mathbf{E}_{\mathbf{W}}=\ker\boldsymbol{\psi}_{\mathbf{W}}$ 
is a family of locally free sheaves on $\p3$.
Moreover, (\ref{E from F}), \eqref{underline E} and (\ref{F from FX}) yield
\begin{equation}\label{E from tilde E}
\mathbf{E}_{\mathbf{W}}\simeq
\underline{\mathbf{E}}_{\mathbf{W}}.
\end{equation}
Furthermore, as the $PGL(N_m)$-principal bundle 
$\tilde{\Phi}:\ \mathbf{W}\to W$ is a categorical factor, and the morphism $\mathbf{f}_V:\mathbf{W}\to V$ by construction is $PGL(N_m)$-invariant, it follows that there exists a morphism
\begin{equation*}
f:\ W\to V
\end{equation*}
such that $\mathbf{f}_V=f\circ\tilde{\Phi}$.
Clearly, $f$ is pointwise just the map $([E],S,C)\mapsto(S,
C,\xi,\mathbf{k}\psi)$ given in \eqref{set-theoretic def of f}.

We have to show that $f$ is an isomorphism. For this, remark 
that the sheaf $\mathbf{D}=pr_{2*}\underline{\mathbf{E}}(m)$, 
where $pr_2:\p3\times V\to V$ is the projection, is a 
locally free $\calo_{V}$-sheaf of rank $N_m$, and the 
evaluation morphism 
$ev:pr_2^*\mathbf{D}\to\underline{\mathbf{E}}(m)$ 
is surjective (see Section \ref{section 2}). Now consider a locally free $\calo_{V}$-sheaf $\mathcal{K}=\mathcal{H}om(\mathbf{k}^{N_m}
\otimes\calo_{V},\mathbf{D})$ and the corresponding scheme 
$\mathbb{V}(\mathcal{K}^{\dual})=\mathrm{Spec}
(\mathrm{Sym}^{\cdot}\mathcal{K}^{\dual})$. There is an open
dense subset $\mathbf{Y}=\mathrm{Isom}(\mathbf{k}^{N_m}
\otimes\calo_{V},\mathbf{D})$ of 
$\mathbb{V}(\mathcal{K}^{\dual})$ consisting
of (fibrewise) invertible homomorphisms from
$\mathbf{k}^{N_m}\otimes\calo_{V}$ to $\mathbf{D}$,
together with the projection $v:\ \mathbf{Y}\to V$ 
and the canonical isomorphism
$\mathrm{can}:\mathbf{k}^{N_m}\otimes\calo_{\p3\times
\mathbf{Y}}
\overset{\simeq}{\to}(\mathrm{id}_{\p3}\times v)^*\mathbf{D}$.
This isomorphism, being twisted by 
$\op3(-m)\boxtimes\calo_{\mathbf{Y}}$,
together with the above epimorphism $ev$ yields an epimorphism
$$
\boldsymbol{\mathcal{B}}\boxtimes\calo_{\mathbf{Y}}\overset{\mathrm{can}}
{\to}(\mathrm{id}_{\p3}\times v)^*\mathbf{D}(-m)
\overset{ev}{\twoheadrightarrow}
\underline{\mathbf{E}}_{\mathbf{Y}},
$$ 
where $\boldsymbol{\mathcal{B}}=\mathbf{k}^{N_m}\otimes\calo_{\p3}
(-m)$ (see Section \ref{section 2}). Thus, by the universal 
property of the open subset $Y$ of the Quot-scheme  
$Q=\mathrm{Quot}_{\p3}(\boldsymbol{\mathcal{B}},P)$ 
introduced in Theorem \ref{Thm 2}(ii), there exists a 
uniquely defined morphism
$\tilde{\mathbf{q}}:\ \mathbf{Y}\to Q$ such that
\begin{equation}\label{EV=}
\underline{\mathbf{E}}_{\mathbf{Y}}\simeq
\mathbb{E}_{\mathbf{Y}},
\end{equation}
where $\mathbb{E}$ is the universal quotient sheaf on 
$\p3\times Q$.
Note that, by (\ref{morphism Phi}),
\begin{equation}
\varphi\circ\tilde{\mathbf{q}}=q\circ v,
\end{equation}
where $\varphi:Y\to N$ is a principal $PGL(N_m)$-bundle (\ref{varphi}).
In particular,
$$
\tilde{\mathbf{q}}(\mathbf{V})\subset Y.
$$

Next, the group $\mathbf{k}^*$ naturally acts on $\mathbf{V}$ by homotheties, so that
\begin{equation}\label{var bf V}
\mathbf{V}:=\mathbf{Y}//\mathbf{k}^*
\end{equation}
is a categorical quotient. Therefore, $v$ as a principal $GL(N_m)$-bundle decomposes as 
$v=\mathbf{\Phi}\circ\nu,$
where $\nu:\mathbf{Y}\to\mathbf{V}$ is a principal $\mathbf{k}^*$-bundle and 
\begin{equation}\label{bf Phi}
\mathbf{\Phi}:\ \mathbf{V}\to V
\end{equation}
is a principal $PGL(N_m)$-bundle.
Since the morphism $\tilde{\mathbf{q}}$ is $\mathbf{k}^*$-invariant it decomposes as 
$$
\tilde{\mathbf{q}}=\mathbf{q}\circ\nu,
$$
where 
\begin{equation}\label{bf q}
\mathbf{q}:\mathbf{V}\to Y 
\end{equation}
is a $PGL(N_m)$-equivarint morphism. Thus, as the principal $PGL(N_m)$-bundles $\mathbf{\Phi}:\ \mathbf{V}\to V$ and $\varphi:\ Y\to N$ are categorical quotients, there exists a morphism $q:\ V\to N$ making the diagram
\begin{equation}\label{diag q}
\xymatrix{ 
Y\ar[d]_{\varphi} & \mathbf{V} \ar[l]_{\mathbf{q}} \ar[d]^{\mathbf{\Phi}} \\
N & V\ar[l]_-{q}}
\end{equation}
cartesian. In addition, similar to (\ref{EV=}) we see that the sheaf $\underline{\mathbf{E}}_{\mathbf{V}}$ satisfies the relation
\begin{equation}\label{EY=}
\underline{\mathbf{E}}_{\mathbf{V}}\simeq
\mathbb{E}_{\mathbf{V}}.
\end{equation}
Note that $\mathbf{V}$ is irreducible, since $V$ is irreducible.

We now construct the morphism 
\begin{equation}\label{mor bf f}
\mathbf{f}:\ \mathbf{W}\to\mathbf{V}
\end{equation}
making the square in the diagram
\begin{equation}\label{diag f,fB}
\xymatrix{ 
\mathbf{W} \ar[d]_{\tilde{\Phi}}\ar[r]^{\mathbf{f}}
\ar[dr]^-{\mathbf{f}_B} & \mathbf{V}\ar[d]^{\boldsymbol{\Phi}} \\
W\ar[r]^-{f} & V}
\end{equation}
cartesian. For this, note that by the universal property of the
Quot-scheme $Q$ the family of generalized null correlation bundles $\mathbf{E}_{\mathbf{W}}$ on $\p3\times\mathbf{W}$ 
defines a morphism 
\begin{equation}\label{eta}
\eta:\ \mathbf{W}\to Q
\end{equation}
such that, by definition, 
$$
\eta(\mathbf{W})\subset Y,
$$ 
and the diagram
\begin{equation}\label{diag xi}
\xymatrix{ 
\mathbf{W} \ar[d]_{\tilde{\Phi}}\ar[r]^{\eta}
 & Y \ar[d]^{\varphi} \\
W\ar[r]^-{q\circ f} & N}
\end{equation}
is cartesian. From the cartesian diagrams (\ref{diag q}) and (\ref{diag xi}) by transitivity of fibred products follows the existence of the desired morphism $\mathbf{f}$ satisfying (\ref{diag f,fB}).

Now consider the composition 
$V\overset{\lambda}{\to}T\overset{\mu}{\to}
\mathrm{R}\overset{r}{\to}\mathrm{P}$ 
of natural morphisms in diagram \eqref{diag for T,B}, and the induced graph of incidence $\Gamma_{B}$
(see \eqref{notation for B}).
Let $\gamma:\ \Gamma_{V}:=\Gamma_{U}\times_UV\to V$ be the projection and set
$$
\mathbf{B}:=(R^2\gamma_*(\underline{\mathbf{E}}
(c-e-4)|_{\Gamma_{V}}))^{\dual}.
$$
A standard base change and the Serre duality (cf. (\ref{h2=h0})) show that $\mathbf{B}$ is
a line bundle on $V$ with a fibre over an arbitrary point
$\mathbf{u}=(S,C,\xi,\mathbf{k}\psi)\in V$ (we use the notaion
from (\ref{B=})) given by
$$
\mathbf{B}\otimes\mathbf{k}(\mathbf{u})=H^0(E(-b)|_S), \ \ \ 
$$
where $E=\underline{\mathbf{E}}|_{\p3\times\{\mathbf{u}\}}$.
Comparing this with (\ref{fibre of G}) and (\ref{U onto G}) and using
(\ref{EY=}) we obtain an epimorphism
$\mathbf{q}^*\mathbb{U}\twoheadrightarrow\mathbf{B}$.
Now by the universal property of 
$\boldsymbol{\mathcal{X}}=\mathbb{P}(\mathbb{U})
\overset{\boldsymbol{\theta}}{\to}\mathcal{Y}$ defined in 
\eqref{bf scr X}
(see, e. g., \cite[Ch. II, Prop. 7.12]{H}) there is a morphism 
$$
\mathbf{g}:\ \mathbf{Y}\to\boldsymbol{\mathcal{X}}
$$
such that $\mathbf{q}=\boldsymbol{\theta}\circ\mathbf{g}$
and $\mathbf{B}\simeq\mathbf{g}^*\calo_{\mathbb{P}(\mathbb{U})}(1)$.
Therefore, in view of (\ref{EY=}) we have 
\begin{equation}\label{EY underline vs EY}
\underline{\mathbf{E}}_{\mathbf{V}}\simeq
(\mathrm{id}_{\p3}\times\mathbf{g})^*\mathbf{E}=
\mathbf{E}_{\mathbf{V}}.
\end{equation}
In addition, since $\mathbf{q}(\mathbf{V})
\subset Y$ by \eqref{bf q}, it follows from diagram \eqref{tilde Y,tilde X} that
\begin{equation*}
\mathbf{g}(\mathbf{Y})\subset\mathbf{X}.
\end{equation*} 
Futhermore, the morphism $\mathbf{g}:\ \mathbf{Y}\to\mathbf{X}$ is an equivariant morphism of principal $PGL(N_m)$-bundles
$\mathbf{\Phi}:\ \mathbf{Y}\to V$ and 
$\Phi:\ \mathbf{X}\to X$. Hence there exists a morphism
$$
g:\ V\to X
$$ 
making the diagram
\begin{equation}\label{g,Phi}
\xymatrix{\mathbf{X}\ar[d]_{\Phi} & \mathbf{V} \ar[l]_{\mathbf{g}} \ar[d]^{\mathbf{\Phi}} \\
X & V\ar[l]_-{g}}
\end{equation}
cartesian. 

We now proceed to constructing the inverse to $f$ morphism
\begin{equation}\label{morphism h}
h:\ V\to W.
\end{equation}
For this, we will first construct the morphism
$$
\mathbf{h}:\ \mathbf{V}\to\mathbf{W}
$$
such that 
\begin{equation}\label{two compns}
\boldsymbol{\pi}\circ\mathbf{h}=\mathbf{g},\ 
\ \ \ \pi\circ h=g\ \ \ \textrm{and}\ \ \ 
\tilde{\Phi}\circ\mathbf{h}=h\circ\Phi,
\end{equation} 
where $\tilde{\Phi}:\mathbf{W}\to W$ is a principal $PGL(N_m)$-bundle in the diagram (\ref{diag W}),
and $\boldsymbol{\pi}$, respectively, $\pi$ are the projections given in that diagram.
Remark that, since the sheaf
$\mathbf{F}_{\mathbf{V}}$ (respectively, the sheaf 
$\underline{\mathbf{F}}_{\mathbf{V}}$) is determined by the sheaf $\mathbf{E}_{\mathbf{V}}$ (respectively, by the sheaf
$\underline{\mathbf{E}}_{\mathbf{V}}$) uniquely up to an isomorphism (see Remark \ref{reduction step}(iii)), the isomorphism (\ref{EY underline vs EY}) implies an isomorphism
\begin{equation*}
\underline{\mathbf{F}}_{\mathbf{V}}\simeq
\mathbf{F}_{\mathbf{V}}.
\end{equation*}
Using this isomorphism, rewrite the left morphism in the exact triple (\ref{univ FT}) twisted by  $\op3(c-a-b)\boxtimes\calo_{\mathbf{T}}$ and lifted onto 
$\p3\times\mathbf{V}$ as
$$
i:\ (\calo_{\p3}\boxtimes\calo_{\mathbf{T}}(1))_{\mathbf{V}}
\to\underline{\mathbf{F}}_{\mathbf{V}}(c-a-b)
\simeq\mathbf{F}_{\mathbf{V}}(c-a-b).
$$
Consider the diagram of natural projections
$$
\xymatrix{
\p3\times\mathbf{X}\ar[d]_{p} & \p3\times\mathbf{V} \ar[l]_{\mathrm{id}_{\p3}\times\mathbf{g}} \ar[d]^{\mathbf{p}} \\
\mathbf{X} & \mathbf{V}\ar[l]_-{\mathbf{g}}}
$$
and apply to the monomorphism $i$ the functor $\mathbf{p}_*$.
We obtain a subbundle morphism
$$
\iota:\ \Theta^*\calo_{\mathbf{T}}(1)\to
\mathbf{p}_*\mathbf{F}_{\mathbf{V}}(c-a-b),\ \ \ \ \ 
\Theta:=\lambda\circ\boldsymbol{\Phi.}
$$ 
Note that 
$\mathbf{p}_*\mathbf{F}_{\mathbf{V}}(c-a-b)$ is a locally
free sheaf (cf. (\ref{rk=})) for which the base change yields
an isomorphism
$$
\mathbf{p}_*\mathbf{F}_{\mathbf{Y}}(c-a-b)\simeq
\mathbf{g}^*p_*\mathbf{F}(c-a-b),
$$ 
hence an epimorphism of locally free sheaves
$$
\mathbf{g}^*(p_*\mathbf{F}(c-a-b))^{\vee}\twoheadrightarrow
\Theta^*\calo_{\mathbf{T}}(-1)
$$
defined as the composition
$$
\epsilon_{\mathbf{V}}:\ \mathbf{g}^*(p_*\mathbf{F}(c-a-b))^{\vee}\simeq
(\mathbf{g}^*p_*\mathbf{F}(c-a-b))^{\vee}\simeq
(\mathbf{p}_*\mathbf{F}_{\mathbf{V}}(c-a-b))^{\vee}
\xrightarrow{\iota^{\vee}}
\Theta^*\calo_{\mathbf{T}}(-1).
$$
Comparing $\epsilon_{\mathbf{Y}}$ with the canonical epimorphism $\epsilon$ in (\ref{can epi}), we obtain by 
the universal property of the projective bundle  $\boldsymbol{\pi}:\mathbf{W}\to\mathbf{X}$ in (\ref{defn W}) that there exists a morphism 
$\mathbf{h}:\ \mathbf{V}\to\mathbf{W}$
satisfying the first relation (\ref{two compns}) and such that
$\mathbf{h}^*\epsilon=\epsilon_{\mathbf{V}}$,\ \  $\mathbf{h}^*\calo_{\mathbf{W}}(1)\simeq
\Theta^*\calo_{\mathbf{T}}(-1)$.
By construction, the morphism $\mathbf{h}$ is $PGL(N_m)$-equivariant, so that it descends to the morphism
$h:\ V\to W$ satisfying the last two relations in (\ref{two compns}).

Remark that, by (\ref{diag f,fB}),  $\mathbf{f}_V=\mathbf{\Phi}\circ\mathbf{f}$.
Therefore, from (\ref{E from tilde E}) 
we obtain 
$\mathbf{E}_{\mathbf{W}}\simeq
(\underline{\mathbf{E}}_{\mathbf{V}})_{\mathbf{W}}=
(\mathrm{id}_{\p3}\times
\mathbf{f})^*\underline{\mathbf{E}}_{\mathbf{V}}$.
This together with (\ref{EY underline vs EY}) yields:
\begin{equation}\label{twice EW}
\mathbf{E}_{\mathbf{W}}\simeq(\mathrm{id}_{\p3}\times
(\mathbf{h}\circ\mathbf{f}))^*\mathbf{E}_{\mathbf{W}}.
\end{equation}
Now a standard argument shows that
\begin{equation}\label{h circ f}
\mathbf{h}\circ\mathbf{f}=\mathrm{id}_{\mathbf{W}}.
\end{equation}
Indeed, consider the Quot-scheme 
\begin{equation}\label{Quot W}
Q_{\mathbf{W}}:=
\mathrm{Quot}_{\p3\times\mathbf{W}/\mathbf{W}}
(\boldsymbol{\mathcal{B}}\boxtimes\calo_{\mathbf{W}},P)\simeq
\mathbf{W}\times Q
\end{equation}
and the embedding
$$
\Delta=(\mathrm{id},\eta):\ \mathbf{W}\to Q_{\mathbf{W}},\ \mathbf{w}\mapsto
(\mathbf{w},\eta(\mathbf{w})),
$$
where the morphism $\eta$ is defined in (\ref{eta}).
Then in view of the universal property of $Q_{\mathbf{W}}$ the relation (\ref{twice EW}) shows that
the composition 
$\mathbf{W}\xrightarrow{\mathbf{h}\circ\mathbf{f}}
\mathbf{W}\xrightarrow{\Delta}Q_{\mathbf{W}}$
coincides with $\Delta$. Hence, since $\Delta$ is an embedding,
(\ref{h circ f}) follows.

Similar to (\ref{h circ f}) one shows that \begin{equation}\label{f circ h}
\mathbf{f}\circ\mathbf{h}=\mathrm{id}_{\mathbf{V}}. 
\end{equation}
(For this, use (\ref{E from tilde E}) to obtain, similar to (\ref{twice EW}), an isomorphism
$\mathbf{E}_{\mathbf{V}}\simeq(\mathrm{id}_{\p3}\times
(\mathbf{f}\circ\mathbf{h}))^*\mathbf{E}_{\mathbf{V}}$,
and then argue as in (\ref{Quot W}), with $Q_{\mathbf{W}}$ 
substituted by $Q_{\mathbf{V}}$, to achieve
(\ref{f circ h}).) From (\ref{h circ f}) and (\ref{f circ h})
it follows that $\mathbf{h}=\mathbf{f}^{-1}$. In particular,
$\mathbf{h}$ is a $PGL(N_m)$-equivariant isomorphism, and we obtain a cartesian diagram of principal $PGL(N_m)$-bundles
\begin{equation}\label{diag bf h, h}
\xymatrix{\mathbf{W}\ar[d]_{\Phi} & \mathbf{Y} \ar[l]_-{\mathbf{h}}^-{\simeq} \ar[d]^{\mathbf{\Phi}} \\
W & B\ar[l]_-{h}.}
\end{equation}
Whence, since $\mathbf{h}$ is an inverse to $\mathbf{f}$, the morphism $h$ is an isomorphism inverse to $f$. Note that $h$ is pointwise just the map $(S,C,\xi,\mathbf{k}\psi)\mapsto
([E],S,C)$ given in \eqref{set-theoretic def of h}.

(ii) Since $W\simeq V$, the stable rationality of $N$ now outcomes from the rationality of $V$ (see Remark \ref{global triple E,F}(ii)) and the local triviality of the $\mathbb{P}^{m}$-fibration $\pi:W\to X$ (Theorem \ref{Thm 4}(ii)) and of the $\mathbb{P}^{\tau}$-fibration $\theta:X\to N$ (Theorem \ref{Thm 2}(i)).
In addition, the isomorphism $f:W\xrightarrow{\simeq}V$ yields 
the desired family $\underline{\mathbf{E}}_{W}=(\mathrm{id}_{\p3}\times f)^*\underline{\mathbf{E}}_{V}$ of generalized null 
correlation bundles on $\p3\times W$ for which  in view of the relation \eqref{E from tilde E} 
the corresponding modular morphism $W\to N$ is 
just the composition of locally trivial projective bundles 
$\pi:W\to X$ and $\theta:X\to N$. 

(iii) From statement (i) and formulas (\ref{dim W tau}) and (\ref{dim X-dim B}) it follows that
\begin{equation}\label{formula for tau}
\tau=\delta(e,a,b,c)+t(e,a,b)-m(e,a,b,c).
\end{equation}
This together with (\ref{t(e)}), (\ref{formula for m}), (\ref{h0(F...)=1}) and (\ref{delta}) shows that, under the conditions $(e,a)\ne(0,0)$, $c>2a+b-e$ and $b>a$, one has
$$
\tau=0.
$$ 
Therefore, by Theorem \ref{Thm 4}(ii) (see \eqref{isom pi}) and Theorem \ref{Thm 2}(i) $W\underset{\simeq}{\xrightarrow{\pi}} X\xrightarrow{\theta}
N(e,a,b,c)$ is a $\mathbb{P}^0$-fibration, hence an isomorphism.
Therefore, by the rationality of $V\simeq W$, $N(e,a,b,c)=
\overline{V}$ is rational. 

In addition, $\underline{\mathbf{E}}_{W}\simeq\mathbf{E}$ is a 
universal family of generalized null correlation bundles over 
$N$. This yields that the scheme $N$ together with the universal 
family $\mathbf{E}$ over it is a fine moduli space in the  sense 
that it represents the functor $F:\mathrm{(Schemes)^0}\to
\mathrm{Sets}$ defined in the following usual way. For a given 
scheme $X$, $F(X)$ is the set of equivalence classes of flat 
families with base $X$ of generalized null correlation bundles 
on $\mathbb{P}^3$ belonging to $N$. Recall that, by definition, 
the two families $\mathcal{E}$ and $\mathcal{E'}$ over $X$ are 
equivalent if they are isomorphic up to a twist by a pullback of 
a line bundle from $X$. Thus, to the equivalence class 
$\{\mathcal{E}\}\in F(X)$ of a family $\mathcal{E}$ there corresponds a morhism $X\to N$ such that $\{\mathcal{E}\}=
\{\mathbf{E}_X\}$. 
\end{proof}

From Theorem \ref{main Thm} and the result of L.~Ein \cite{Ein} now immediately follows
\begin{corollary}\label{Cor}
For both $e=0$ and $e=-1$, the union of the spaces $M(e,n)$ over all $n\ge1$ contains an infinite series of rational components.
\end{corollary}
The following remarks are in order.

\begin{remark}\label{fine moduli}
Fine moduli for $n$ even. There is a well-known sufficient 
condition for the (given component of the) Gieseker-Maruyama 
moduli space to be fine -- see \cite[Cor. 4.6.6]{HL}. In case 
of $M(0,n)$ with $n$ even this condition fails, and there were 
no known examples of components of
$M(0,n)$ when these moduli components were fine moduli spaces. (On the contrary, there are known certain components of 
$M(0,n)$ for $n$ even, e. g., the instanton components which are not fine -- see \cite{HN}.) Theorem \ref{main Thm}(ii) provides a series of fine (open dense subsets of) moduli components $N(e,a,b,c)$ for $c>2a+b-e,\ b>a,\ (e,a)\ne(0,0)$, and $n=c^2-a^2-b^2$ even, this series clearly being infinite -- see \cite{KOT}.
\end{remark}

\begin{remark}\label{V(k,l)}
In 1984 V.~K.~Vedernikov \cite{V1} constructed, 
for $1\le l\le k$, a family $V_1(k,l)\subset 
M(0,2kl+2l-l^2)$;  for $1\le2l\le k$, a family 
$V_2(k,l)\subset M(0,k^2+2k+1-l^2)$; for $1\le2l\le k+2$, a 
family $V_3(k,l)k^2+3k+2+2l-2l^2)$. Later in 1987 (see 
\cite{V2}), he constructed one more family, $V_4(k)\subset 
M(0,(k+1)^2)$ for $k\ge1$.
From the results of L.~Ein, 1988, see \cite{Ein}, it follows that Ein components $N(e,a,b,c)$ with approriate $a,b,c$ contain these Vedernikov's families $V_1(k,l)$ and $V_4(k)$, respectively,  $V_2(k,l)$ and  $V_3(k,l)$, as their open dense subsets in special cases when $e=a=0$, respectively, $a=b$. More precisely,
\begin{equation}\label{Ein=Ved}
\begin{split}
& \overline{V_1(k,l)}=N(0,a,b,c)\ \ \ \ \ {for}\ \ \  a=0,\ \ b=k+1-l,\ \ c=k+1,\\
& \overline{V_2(k,l)}=N(0,a,b,c)\ \ \ \ \ {for}\ \ \  a=b=l,\ \ c=k+1,\\
& \overline{V_3(k,l)}=N(-1,a,b,c)\ \ \ {for}\ \ \  a=b=l-1,\ \ c=k+1,\\
& \overline{V_4(k)}=N(0,a,b,c)\ \ \ \ \ \ \ {for}\ \ \  a=b=0,\ \ \ c=k+1.\\
\end{split}
\end{equation}
In \cite{V1}, it is asserted that $V_1(k,l)$ is rational.
However, the construction of rationality of $V_1(k,l)$ presented in \cite[Section 3]{V1} coincides with ours and thus, by Theorem \ref{main Thm}, yields only stable rationality of $V_1(k,l)$. Indeed, in this case, $\tau=0$ by \eqref{formula for tau}, but $m=m(0,0,k+1-l,k+1)=1$ by \eqref{formula for m} and \eqref{h0(F...)=1}, so that, $\pi:B_{\tau}\to V_1(k,l)$ is a locally trivial $\mathbb{P}^1$-bundle with $B_{\tau}$ rational. So the problem of rationality of $V_1(k,l)$ remains open.

The construction of rationality of $V_2(k,l)$ provided
in \cite[Sections 5-6]{V1} differs from ours. According to Theorem \ref{main Thm}, the rationality of $V_2(k,l)$ is covered by our result in the range $k\ge 3l\ge3$ and, respectively, not covered in the range $2\le2l\le k\le3l-1$.

In \cite[Section 7]{V1}, the rationality of $V_3(k,l)$ is asserted without proof. On the other hand, in this case the rationality (respectively, stable rationality) of $V_3(k,l)$ follows from Theorem \ref{main Thm} for $k\ge3l-2$ (respectively, for $2l-2\le k\le3l-3$). 

Last, the rationality of $V_4(k)$ is proved in \cite{V2}. It is not covered by Theorem \ref{main Thm}. Indeed, in this case we obtain from \eqref{formula for m} and \eqref{h0(F...)=1} that $m=2$, and Theorem \ref{main Thm} yields stable rationality of $V_4(k)$.

Summarizing the above and using \eqref{Ein=Ved}, we conclude that the result of Theorem \ref{main Thm} covers Vedernikov's (proven) results in case $e=0,\ a=b>0,\ c>3a$ and improves them in case $e=a=0,\ b>0$. 
\end{remark}
\begin{remark}\label{big families}
As it is known \cite[Prop. 3.1]{Rao}, \cite{Ein}, the 
cohomology module $H_*^1(E)$ of a generalized null correlation 
bundle $[E]\in N_{\mathrm{nc}}$ has one generator as 
a graded module over $\mathbf{k}[x_0,x_1,x_2,x_3]$. Using 
this, A.~P.~Rao in \cite[Prop. 3.1 and Remark 3.2]{Rao} 
constructed big enough rational families $B$ of generalized 
null correlation bundles from $N_{\mathrm{nc}}$ with 
a given cohomology module $H_*^1(E)$. It follows that 
$N_{\mathrm{nc}}$ can be filled by unirational varieties $\Phi(B)$ of dimension big enough, where $\Phi:\ B\to 
N_{\mathrm{nc}}$ is the modular morphism. This shows 
that $N_{\mathrm{nc}}$ is at least rationally connected (which also follows from their stable rationality), and it possibly might give an alternative approach to the problem of rationality of Ein components.
\end{remark}

\section{Components of the moduli space $M(e,n)$ for small $n$}\label{section 8}

\vspace{2mm}

In this section, we enumerate the known components (including the Ein components) of the Gieseker-Maruyama moduli space  $M(e,n)$ for small values of  $n$, namely, for $n\le20$ in both cases (i) $e=0$ and (ii) $e=-1$. We specify those of these components which are rational, respectively, stably rational. Their dimensions are also given. 

\vspace{2mm}
(i) $e=0$. The complete description of all the components of $M(0,n)$ is currently known only for $n\le5$.

\vspace{2mm}
(i.1) $M(0,1)$ is irreducible: $M(0,1)\simeq\mathbb{P}^5\setminus G$, where $G$ is the Grassmannian $Gr(2,4)$ embedded in
$\mathbb{P}^5$ by Pl\"ucker -- see, e.g., \cite{{H-vb}} or \cite{OSS}. Here $M(0,1)$ is an Ein component with $a=b=0,\ c=1$.

(i.2) $M(0,2)$ is an irreducible 13-dimensional rational 
variety, and any sheaf in $M(0,2)$ is an instanton bundle -- 
see \cite[Section 9]{H-vb}. Note that $M(0,2)$ is not an Ein 
component. 
 
(i.3) $M(0,3)$ consists of two rational irreducible 21-dimensional components: the instanton component $I_3$ any sheaf of which is an instanton bundle, and the Ein component $\overline{N}(0,0,1,2)$ any sheaf of which is a generalized null correlation bundle, i. e. $\overline{N}(0,0,1,2)=\overline{N}(0,0,1,2)^{\mathrm{nc}}$ -- see
\cite{ES}.

(i.4) $M(0,4)$ consists of two irreducible 29-dimensional components: the instanton component $I_4$ any sheaf of which is a mathematical instanton bundle with spectrum $(0,0,0,0)$, and the Ein component $\overline{N}(0,0,0,2)$ -- see \cite{B}, \cite{B1}, \cite{Ch}, \cite{HR}. The rationality of $\overline{N}(0,0,0,2)$ is proved in \cite{Ch} and reproved in  \cite{V2} by another method. It is also shown in \cite{Ch} that $\overline{N}(0,0,0,2)\setminus N_{\mathrm{nc}}\ne\emptyset$.

(i.5) $M(0,5)$ has three irreducible components, according to a recent result of C.~Almeida, M.~Jardim, A.~Tikhomirov and S.~Tikhomirov \cite{AJTT}. The first one is the 37-dimensional rational instanton component $I_5$ \cite{CTT}, \cite{T1}, \cite{K}, a general sheaf of which is a mathematical instanton bundle.
The next one is the 40-dimensional Ein component $\overline{N}(0,0,2,3)$  -- see \cite{Ein}, \cite[Theorem 4.7]{ES}, \cite{HR}, and it coincides with the component $Q_2$ of $M(0,5)$ introduced by
Ellingsrud and Str\o mme (we use the notation from Section \ref{section 1}). This component is stably rational by Theorem \ref{main Thm}. (A weaker statement about unirationality of $\overline{N}(0,0,2,3)=Q_2$ was mentioned in Section \ref{section 1}.) 
The third one is a 37-dimensional component $M_b$ described as the closure in $M(0,5)$ of the set $\{[E]\in M(0,5)\ |\ E$ is a cohomology bundle of a monad of the type $0\to\op3(-2)\oplus\op3(-1)\to 6\op3\to\op3(1)\oplus\op3(2)\to0\}$.

(i.6) $M(0,6)$ contains the instanton component $I_6$ of dimension 45 (see \cite{T2}) and at least one more component of
dimension $\ge45$ which contains a (possibly open) locally closed subset $M_6=\{[E]\in M(0,6)\ |\ E$ is the cohomology bundle of a monad $0\to2\op3(-1)\oplus\op3(-2)\to8\op3\to2\op3(1)
\oplus\op3(2)\to0\}$ -- see \cite[Table 5.3, $c_2=6$,\ (2,i)]{HR},
where $\dim M_6=45$ by Barth's formula \cite[p. 216]{B}. However,
$M(0,6)$ does not contain Ein components, since there are no integer solutions for $a,b,c$ satisfying the conditions $b\ge a\ge0,\ c>a+b,\ c^2-a^2-b^2=6$ -- see \cite[Section 2]{KOT}.

(i.7) $M(0,7)$ contains at least four irreducible components. They are: the instanton component $I_7$ of dimension 53 \cite{T1}, the two Ein components $\overline{N}(0,0,3,4)$ and $\overline{N}(0,1,1,3)$ of dimensions 65 and 55, respectively, and a component of dimension $\ge53$ which contains a locally closed subset $M_7=\{[E]\in M(0,7)\ |\ E$ is the cohomology bundle of a monad $0\to3\op3(-1)\oplus\op3(-2)\to10\op3\to3\op3(1)
\oplus\op3(2)\to0\}$ -- see \cite[Table 5.3, case $c_2=7$,\ (2,i)]{HR}, where $\dim M_7=52$ by Barth's formula [loc. cit.]. Here the Ein components $\overline{N}(0,0,3,4)$ and $\overline{N}(0,1,1,3)$ are stably rational by Theorem \ref{main Thm}, and there are no other Ein components in $M(0,7)$ by \cite[Section 2]{KOT}.

(i.8) $M(0,8)$ contains at least three irreducible components. They are: the instanton component $I_8$ of dimension 61 \cite{T2}, the Ein component $\overline{N}(0,0,1,3)$ of dimension 62, and a component of dimension $\ge61$ which contains a (possibly open) locally closed subset $M_8=\{[E]\in M(0,8)\ |\ E$ is the cohomology bundle of a monad 
$0\to4\op3(-1)\oplus\op3(-2)\to12\op3\to4\op3(1)\oplus\op3(2)\to
0\}$ -- see \cite[Table 5.3, case $c_2=8$,\ (2,i)]{HR},
where $\dim M_8=61$ by Barth's formula. Here the Ein
component $\overline{N}(0,0,1,3)$ is stably rational by Theorem \ref{main Thm}, and there are no other Ein components in $M(0,8)$ by \cite[Section 2]{KOT}.

We complete, using \cite[Main Theorem 1]{AJTT}, \cite[Section 2]{KOT} and \cite[Theorem 3]{TTV}, the list of all
currently known irreducible components of $M(0,n)$ for $9\le 
n\le20$. For these values of $n$, besides the Ein components 
and the instanton components $I_n$ of dimension $8n-3,\ 9\le 
n\le 20$, (the rationality or stable rationality of these 
$I_n$'s is unknown), the known irreducible components are 6 
more components. They are: 
1) component of dimension 69 of $M(0,9)$, 2) component of dimension 77 of $M(0,10)$, 3) component of dimension 85 of $M(0,11)$, 4) component of dimension 93 of $M(0,12)$,
5) component of dimension 135 of $M(0,17)$, 6) component of dimension 141 of $M(0,18)$. 

Below we list the Ein components of $M(0,n)$ for $9\le n\le20$. Their rationality or stable rationality follows 
from Theorem \ref{main Thm} and Remark \ref{V(k,l)}, and their 
dimensions are given by (\ref{dim N 1}).\\
$n=9:$\ $\overline{N}(0,0,0,3)$ rational of dimension 69, $\overline{N}(0,0,4,5)$ stably rational of dimension 96;\\
$n=10:$\ no Ein components;\\ 
$n=11:$\ $\overline{N}(0,0,5,6)$ stably rational of dimension 133, $\overline{N}(0,1,2,4)$ stably rational of dimension 98;\\
$n=12:$\ $\overline{N}(0,0,2,4)$ stably rational of dimension 104;\\
$n=13:$\ $\overline{N}(0,0,6,7)$ stably rational of dimension 176;\\
$n=14:$\ $\overline{N}(0,1,1,4)$ rational of dimension 117;\\ 
$n=15:$\ $\overline{N}(0,0,1,4)$ stably rational of dimension 123, $\overline{N}(0,0,7,8)$ stably rational of dimension 225, $\overline{N}(0,1,3,5)$ stably rational of dimension 152;\\
$n=16:$\ $\overline{N}(0,0,0,4)$ rational of dimension 129, $\overline{N}(0,0,3,5)$ stably rational of dimension 158;\\ 
$n=17:$\ $\overline{N}(0,0,8,9)$ stably rational of dimension 280, $\overline{N}(0,2,2,5)$ stably rational of dimension 170;\\ 
$n=18:$\ no Ein components;\\ 
$n=19:$\ $\overline{N}(0,0,9,10)$ stably rational of dimension 341, 
$\overline{N}(0,1,4,6)$ stably rational of dimension 218;\\ 
$n=20:$\ $\overline{N}(0,0,4,6)$ stably rational of dimension 224, 
$\overline{N}(0,1, 2, 5)$ rational of dimension 187.

\vspace{3mm}

(ii) $e=-1$. The scheme $M(-1,n)$ is known to be nonempty only for $n=2m,\ m\ge1$ \cite{H-r}. Moreover,
Hartshorne in \cite{H-r} produced a family $H_m$ of bundles  with minimal spectrum from $M(-1,2m)$, using the Serre construction similar to that of 'tHooft instanton bundles from $I_m$. (For the notion of spectrum see \cite[Section 7]{H-r}.) Hartshorne showed that, for each $m$, the family $H_m$ is contained in a unique irreducible $(16m-5)$-dimensional component of $M(-1,2m)$ which is smooth along $H_m$. Denote this component by $Y_{2m}$. 

Now observe the spaces $M(-1,2m)$ for $m=1,2,3$.

(ii.1) $M(-1,2)=Y_2$ is an irreducible rational variety of dimension 11 \cite{HS}.

(ii.2) $M(-1,4)$ has two irreducible components: the rational component $Y_4$ of dimension 27, and the rational component $M$ of dimension 28 which consists of bundles with maximal spectrum \cite{BM}.

(ii.3) $M(-1,6)$ has at least three irreducible components: the component $Y_3$ of the expected dimension 43;
the Ein component $\overline{N}(-1,0,0,2)$ which, by Theorem \ref{main Thm}, is a rational variety of the expected dimension 43;
the Ein component $\overline{N}(-1,0,2,3)$ which, by Theorem \ref{main Thm}, is a stably rational variety of dimension 50. Note that these two Ein components differ by the spectra of bundles therein (see \cite{Ts}). Besides, as it follows from \cite{KOT}, there are no other Ein components in $M(-1,6)$.

We complete  the list of all known irreducible components of
$M(-1,n)$ for $8\le n\le20$, $n$ even. Besides the components $Y_n$ of dimension $8n-5$, the rationality or stable rationality of which is unknown, these are Ein components of $M(-1,n)$. (As above, here \cite[Section 2]{KOT}, Theorem \ref{main Thm}, Remark \ref{V(k,l)}, and  (\ref{dim N 1}) are used.)\\
$n=8:$\ $\overline{N}(-1,0,3,4)$ stably rational of dimension 78,
$\overline{N}(-1,1,1,3)$ stably rational of dimension 67;\\
$n=10:$\ $\overline{N}(-1,0,1,3)$ rational of dimension 80, 
  $\overline{N}(-1,0,4,5)$ stably rational of dimension 112;\\
$n=12:$\ $\overline{N}(-1,0,0,3)$ rational of dimension 93, 
 $\overline{N}(-1,0,5,6)$ stably rational of dimension 152, 
 $\overline{N}(-1,1,2,4)$ stably rational of dimension 116;\\
$n=14:$\ $\overline{N}(-1,0,2,4)$ rational of dimension 128, 
  $\overline{N}(-1,0,6,7)$ stably rational of dimension 198;\\
$n=16:$\ $\overline{N}(-1,0,7,8)$ stably rational of dimension 250, 
  $\overline{N}(-1,1,1,4)$ stably rational of dimension 143, 
  $\overline{N}(-1,1,3,5)$ stably rational of dimension 176;\\ 
$n=18:$\ $\overline{N}(-1,0,1,4)$ rational of dimension 154, 
 $\overline{N}(-1,0,3,5)$ rational of dimension 188, 
 $\overline{N}(-1,0,8,9)$ stably rational of dimension 308, 
  $\overline{N}(-1,2,2,5)$ stably rational of dimension 197;\\
$n=20:$\ $\overline{N}(-1,0,0,4)$ rational of dimension 165,
  $\overline{N}(-1,0,9,10)$ stably rational of dimension 372, 
  $\overline{N}(-1,1,4,6)$ stably rational of dimension 248.

\end{document}